\documentclass[12pt]{amsart}
\usepackage[shortlabels]{enumitem}
\usepackage{SSdefn}

\newcommand{\FI}{\mathbf{FI}}
\newcommand{\gen}{\mathrm{gen}}
\newcommand{\alg}{\mathrm{alg}}
\newcommand{\bone}{\mathbf{1}}

\DeclareMathOperator{\Sh}{Sh}
\DeclareMathOperator{\sh}{sh}
\DeclareMathOperator{\LEx}{LEx}
\DeclareMathOperator{\magn}{magn}
\DeclareMathOperator{\REP}{REP}
\DeclareMathOperator{\Isom}{Isom}
\newcommand{\umu}{\ul{\smash{\mu}}}
\newcommand{\unu}{\ul{\smash{\nu}}}

\newcommand{\usigma}{\ul{\smash{\sigma}}}
\newcommand{\utau}{\ul{\smash{\tau}}}
\newcommand{\lf}{\mathrm{lf}}
\newcommand{\ulambda}{\ul{\smash{\lambda}}}

\newcommand{\stacks}[1]{\cite[\href{http://stacks.math.columbia.edu/tag/#1}{Tag~#1}]{stacks}}

\title[{Stable representation theory: beyond the classical groups}]{Stable representation theory:\\ beyond the classical groups}

\author{Andrew Snowden}
\address{Department of Mathematics, University of Michigan, Ann Arbor, MI}
\email{\href{mailto:asnowden@umich.edu}{asnowden@umich.edu}}
\urladdr{\url{http://www-personal.umich.edu/~asnowden/}}

\thanks{AS was supported by NSF DMS-1453893.}

\date{September 23, 2021}

\begin{document}

\begin{abstract}
The orthogonal groups are a series of simple Lie groups associated to symmetric bilinear forms. There is no analogous series associated to symmetric trilinear forms. We introduce an infinite dimensional group-like object that can be viewed as the limit of this non-existent series, were it to exist. We show that the representation theory of this object is well-behaved, and similar to the stable representation theory of orthogonal groups. Our theory is not specific to symmetric trilinear forms, and applies to any kind of tensorial forms. Our results can be also be viewed from the perspective of semi-linear representations of the infinite general linear group, and are closely related to twisted commutative algebras.
\end{abstract}

\maketitle
\tableofcontents

\section{Introduction}

Bilinear forms are remarkable objects: they have just the right amount of complexity to be tractable and yet still interesting. Their symmetry groups, the orthogonal and symplectic groups, are among the most important objects in mathematics. Trilinear forms, on the other hand, are too complicated. Their symmetry groups are diverse, but generically finite, and do not give rise to new families of simple Lie groups.

It has recently been discovered \cite{bde,polygeom,des} that, somewhat surprisingly, trilinear forms (and higher degree tensorial forms) in infinite dimensions are \emph{less} complicated than their finite dimensional counterparts, and more like bilinear forms. In particular, up to a certain notion of equivalence, there is a unique non-degenerate form of each type (e.g., symmetric trilinear). The purpose of this paper is to introduce group-like objects (called \emph{germinal subgroups}) that capture the symmetry of these forms, and to study their representation theory. We find that this representation theory is very well-behaved, and closely parallels the stable representation theory of the classical groups. Thus, while there is not a family of simple Lie groups attached to, say, symmetric trilinear forms, there is nonetheless a reasonable limiting object.

\subsection{Generalized orbits and stabilizers} \label{ss:intro-gen-orb}

We explain our main ideas and results in the setting of symmetric trilinear forms over the complex numbers to keep the exposition simple. We work more generally in the body of the paper.

Let $X_n=\Sym^3(\bC^n)^*$ be the space of symmetric trilinear forms (i.e., cubic polynomials) in $n$ variables. Also, let $\bC^{\infty}=\bigcup_{n \ge 1} \bC^n$ and $X_{\infty}=\Sym^3(\bC^{\infty})^*$. The set $X_{\infty}$ is the inverse limit of the sets $X_n$, and as such carries the inverse limit topology. (Each $X_n$ is endowed with the discrete topology). Precisely, a sequence $\{\omega_i\}_{i \ge 1}$ in $X_{\infty}$ converges to $\omega$ if for each $n$ we have $\omega_i \vert_{\bC^n} = \omega \vert_{\bC^n}$ for all sufficiently large $i$.

The group $\GL_n(\bC)$ acts on $X_n$, and the group $\GL=\bigcup_{n \ge 1} \GL_n(\bC)$ acts on $X_{\infty}$. The group $\GL$ is, in a sense, too small\footnote{The group $\Aut(\bC^{\infty})$ is much larger than $\GL$, but it is also too small.}. To remedy this, we employ a modification of the concept of orbit: we say that two elements of $X_{\infty}$ belong to the same \defn{generalized orbit} if each belongs to the closure of the orbit of the other. This idea was introduced in a slightly different way in \cite{polygeom}; see \S \ref{ss:glvar} and Remark~\ref{rmk:gen-orb} for details.

We say that an element of $X_{\infty}$ is \defn{degenerate} if it has the form $\sum_{i=1}^n q_i \ell_i$ where $q_i \in \Sym^2(\bC^{\infty})^*$ and $\ell_i \in (\bC^{\infty})^*$, and \defn{non-degenerate} otherwise. The main theorem of \cite{des} asserts that the non-degnerate forms constitute a single generalized orbit. (The paper \cite{des} concerns only symmetric trilinear forms, but this statement was extended to other types of tensorial forms in \cite{bde,polygeom}.)

Just as the usual orbits of $\GL$ are too small, so too are the usual stabilizers. One question we sought to answer in this paper is: what is the right notion of ``generalized stabilizer''? We have come to the following idea. Let $\omega \in X_{\infty}$ be given. For $n \ge 1$, define $\Gamma_{\omega}(n)$ be the set of elements $g \in \GL$ such that $g^{-1}\omega \vert_{\bC^n} = \omega \vert_{\bC^n}$. (The inverse here is simply to make some other definitions cleaner.) Note that $\Gamma_{\omega}(n)$ is typically not a subgroup. We define the \defn{generalized stabilizer} of $\omega$ to be the system $\Gamma_{\omega}=\{\Gamma_{\omega}(n)\}_{n \ge 1}$. One should think of $\Gamma_{\omega}$ as a kind of germ of a neighborhood of the stabilizer of $\omega$. For this reason, we refer to $\Gamma_{\omega}$ as a \emph{germinal subgroup}; see Definition~ \ref{def:asub} for details. We view $\Gamma_{\omega}$ as an analog of the infinite orthogonal group associated to symmetric trilinear forms.

\subsection{Representations of generalized stabilizers} \label{ss:repgs}

Let $\omega \in X_{\infty}$ and $\Gamma_{\omega}$ be as above. We define a \defn{representation} of $\Gamma_{\omega}$ to be a complex vector space $V$ such that each finite dimensional subspace $W \subset V$ is endowed with an action map $\Gamma_{\omega}(n) \times W \to V$, for some $n$ depending on $W$, satisfying certain conditions. A little more precisely, the data defining a representation can be encoded as a linear map
\begin{displaymath}
V \to \varinjlim_{n \to \infty} \Fun(\Gamma_{\omega}(n), V).
\end{displaymath}
Thus for $v \in V$, one can regard $g \mapsto gv$ as the germ of a function on $\GL$, with respect to the system of neighborhoods $\Gamma_{\omega}$.

Every representation of $\GL$ restricts to a representation of $\Gamma_{\omega}$. We say that a representation of $\Gamma_{\omega}$ is \defn{algebraic} if it occurs as a subquotient of the restriction of a polynomial representation of $\GL$. In particular, the standard representaiton $\bC^{\infty}$ of $\GL$ restricts to an algebraic representation of $\Gamma_{\omega}$, which we call the standard representation of $\Gamma_{\omega}$. We let $\Rep^{\alg}(\Gamma_{\omega})$ denote the category of algebraic representations. This is a Grothendieck abelian category equipped with a tensor product.

The primary purpose of this paper is to understand the algebraic representation theory of $\Gamma_{\omega}$ when $\omega$ is non-degenerate. The following is a summary of our findings.
\begin{itemize}
\item Algebraic representations enjoy several finiteness properties:
\begin{itemize}
\item Every algebraic representation is the union of its finite length subrepresentations.
\item The tensor product of two finite length algebraic representations is again finite length.
\item If $V$ and $W$ are finite length algebraic representations then $\Hom_{\Gamma_{\omega}}(V, W)$ is a finite dimensional complex vector space.
\end{itemize}
\item The simple algebraic representations are well-understood:
\begin{itemize}
\item For each partition $\lambda$, there is a simple $L_{\lambda}$, and these exhaust the simples.
\item One can construct $L_{\lambda}$ using a variant of Weyl's traceless tensor construction. Let $T^{[n]}$ be the intersection of the kernels of the maps $(\bC^{\infty})^{\otimes n} \to (\bC^{\infty})^{\otimes (n-3)}$ obtained by applying $\omega$ to three tensor factors. This space carries an action of $\fS_n \times \Gamma_{\omega}$, where $\fS_n$ denotes the symmetric group. The isotypic piece of $T^{[n]}$ corresponding the Specht module $S^{\lambda}$ is exactly $L_{\lambda}$.
\end{itemize}
\item Algebraic representations are well-behaved homologically:
\begin{itemize}
\item The representations $\bS_{\lambda}(\bC^{\infty})$ are exactly the indecomposable injective algebraic representations; in fact, $\bS_{\lambda}(\bC^{\infty})$ is the injective envelope of $L_{\lambda}$.
\item Every finite length algebraic representation has finite injective dimension.
\end{itemize}
\item There is a combinatorial description of the entire category $\Rep^{\alg}(\Gamma_{\omega})$: it is equivariant to a category of representations of a certain variant of the upwards Brauer category.
\item The category $\Rep^{\alg}(\Gamma_{\omega})$ satisfies a universal property. Let $\cC$ be a $\bC$-linear abelian category equipped with a tensor product. Then giving a left-exact symmetric monoidal $\bC$-linear functor $\Rep^{\alg}(\Gamma_{\omega})^{\rf} \to \cC$ is equivalent to giving an object of $\cC$ equipped with a symmetric trilinear form. The notation $(-)^{\rf}$ here denotes the subcategory of finite length objects.
\item The symmetric monoidal category $\Rep^{\alg}(\Gamma_{\omega})$ is independent of $\omega$, up to equivalence.
\end{itemize}

\subsection{Semi-linear representations: motivation}

Recall that if a group $G$ acts on a field $K$ then a \defn{semi-linear representation} of $G$ over $K$ is a $K$-vector space $V$ equipped with an additive action of $G$ such that the equation $g(av)=(ga)(gv)$ holds, for $g \in G$, $a \in K$, and $v \in V$. Semi-linear representations will be a central topic in this paper. To motivate their appearance, we first examine a familiar case.

Let $Y_n=\Sym^2(\bC^n)^*$ be the space of symmetric bilinear forms on $\bC^n$, regarded as an algebraic variety; explicitly, $Y_n=\Spec(S_n)$ where $S_n$ is the polynomial ring $\Sym(\Sym^2(\bC^n))$. Let $Y_n^{\circ}$ be the open subvariety of $Y_n$ consisting of non-degenerate forms. The algebraic group $\GL_n$ acts transitively on $Y_n^{\circ}$. Let $y$ be a closed point of $Y_n^{\circ}$, and let $\bO_n$ be its stabilizer. If $\cF$ is a $\GL_n$-equivariant quasi-coherent sheaf on $Y_n^{\circ}$ then its fiber $\cF(y)$ at $y$ is an algebraic representation of $\bO_n$, and this construction gives an equivalence of categories
\begin{displaymath}
\QCoh(Y_n^{\circ})^{\GL_n} \to \Rep(\bO_n)
\end{displaymath}
In fact, we can get a similar equivalence using the generic point of $Y_n^{\circ}$. If $\cF$ is a $\GL_n$-equivariant quasi-coherent sheaf on $Y_n^{\circ}$ then its generic fiber is a semi-linear representation of $\GL_n$ over $\Frac(S_n)$ that is algebraic (in the sense that it is spanned by an algebraic subrepresentation). Moreover, letting $\cC_n$ be the category of such semi-linear representations, this construction defines an equivalence
\begin{displaymath}
\QCoh(Y_n^{\circ})^{\GL_n} \to \cC_n
\end{displaymath}
Thus, combined with the previous equivalence, we obtain an equivalence
\begin{displaymath}
\Rep(\bO_n) = \cC_n.
\end{displaymath}
This gives us a way of studying representations of $\bO_n$ (or, at least, the representation category) even if we do not understand the group $\bO_n$ very well.

We adopt this approach in this paper to replace representations of $\Gamma_{\omega}$ with more familiar objects. Let $R$ be the infinite variable polynomial ring $\Sym(\Sym^3(\bC^{\infty}))$ and let $K=\Frac(R)$. We show (Theorem~\ref{thm:genstab}) that $\Rep^{\alg}(\Gamma_{\omega})$ is equivalent to a certain category of semi-linear representations of $\GL$ over $K$ (precisely, the category of ``$K$-modules'' introduced below). The proof is similar to the one outlined above, but technically more involved, and relies on some non-trivial results from \cite{universality} and \cite{polygeom}. We find the semi-linear perspective to be technically much easier to work with, so most of the paper is carried out in this setting.

\subsection{Semi-linear representations: results}

We now explain some of our results on semi-linear representations in more detail. We first introduce some fundamental definitions. A \defn{$\GL$-algebra}\footnote{In characteristic~0, $\GL$-algebras are equivalent, under Schur--Weyl duality, to twisted commutative algebras; see \cite[\S 8.1]{expos}.} is an algebra object in the category of polynomial representations of $\GL$; in other words, it is a commutative ring equipped with an action of $\GL$ under which it forms a polynomial representation. For example, the ring $\Sym(\Sym^3(\bC^{\infty}))$ appearing above is a $\GL$-algebra. If $R$ is a $\GL$-algebra then an \defn{$R$-module} is a module object; in other words, it is a $\GL$-equivariant $R$-module that forms a polynomial representation.

A \defn{$\GL$-field} is a field equipped with an action of $\GL$ that can be obtained as the fraction field of an integral $\GL$-algebra. If $K$ is a $\GL$-field then a \defn{$K$-module} is a semi-linear representation of $\GL$ over $K$ that is generated by a polynomial subrepresentation. The basic example of a $K$-module is $K^{\oplus \lambda} = K \otimes_{\bC} \bS_{\lambda}(\bC^{\infty})$. While $K^{\oplus \lambda}$ is typically not projective, every $K$-module is a quotient of a sum of ones of this form. We let $\Mod_K$ denote the category of $K$-modules. This is the fundamental object of study in this paper.

We prove two main technical results about $K$-modules. To state the first one, we must introduce the shift operation. Let $G(n)$ be the subgroup of $\GL$ consisting of block matrices of the form
\begin{displaymath}
\begin{pmatrix} 1 & 0 \\ 0 & \ast \end{pmatrix}
\end{displaymath}
where the top left block is $n \times n$. This group is isomorphic to $\GL$. If $X$ is a set equipped with an action of $\GL$, we define its $n$th \defn{shift}, denoted $\Sh_n(X)$, to be the set $X$ equipped with the action of $\GL$ coming from restricting the given action to $G(n) \cong \GL$. The shift operation preserves all structure introduced so far (polynomial representations, $\GL$-fields, etc.). Our first theorem is:

\begin{theorem}[Shift theorem]
Let $K$ be a $\GL$-field and let $M$ be a finitely generated $K$-module. Then there exists $n \ge 0$ and partitions $\lambda_1, \ldots, \lambda_r$ such that $\Sh_n(M)$ is isomorphic to $\bigoplus_{i=1}^r \Sh_n(K)^{\oplus \lambda_i}$ as a $\Sh_n(K)$-module.
\end{theorem}

This theorem is an instance of the general principle in representation stability that objects can be made ``nice'' after shifting. The first theorem of this sort was Nagpal's shift theorem for $\FI$-modules \cite{nagpal}. The above shift theorem is closely related to the shift theorem for $\GL$-varieties \cite[Theorem~5.1]{polygeom}, and follows a similar proof.

Our second main result about $K$-modules is the following:

\begin{theorem}[Embedding theorem]
Let $K$ be a rational $\GL$-field, i.e., one of the form $\Frac(\Sym(E))$ where $E$ is a finite length polynomial representation of $\GL$, and let $M$ be a finitely generated $K$-module. Then there exist partitions $\lambda_1, \ldots, \lambda_r$ and an injection of $K$-modules $M \to \bigoplus_{i=1}^r K^{\oplus \lambda_i}$.
\end{theorem}

This theorem follows rather easily from the shift theorem. It is a very important theorem for us: indeed, all the statements in \S \ref{ss:repgs} have analogs for $\Mod_K$, and can be deduced from the embedding thoerem by comparitively easy arguments. The corresponding results for $\Rep^{\alg}(\Gamma_{\omega})$ are deduced from those for $\Mod_K$.

\subsection{Summary of categories}

Let $\fU$ be the analog of the upwards Brauer category for symmetric trilinear forms (see \S \ref{ss:brauer}), let $R=\Sym(\Sym^3(\bC^{\infty}))$, let $K=\Frac(R)$, and let $\omega \in \Sym^3(\bC^{\infty})^*$ be non-degenerate. We show that the following categories are equivalent:
\begin{enumerate}
\item The category $\Mod_{\fU}^{\lf}$ of $\fU$-modules that are locally of finite length.
\item The category $\Mod_R^{\lf}$ of $R$-modules that are locally of finite length.
\item The generic category $\Mod_R^{\gen}$, i.e., the Serre quotient of $\Mod_R$ by the subcategory of torsion modules.
\item The category $\Mod_K$.
\item The category $\Rep^{\alg}(\Gamma_{\omega})$.
\end{enumerate}
The equivalence between (a) and (b) is straightforward, as is the equivalence between (c) and (d). The equivalence of (b) and (c) is much more difficult, and relies upon the embedding theorem. The equivalence between (d) and (e) is also difficult, and relies on non-trivial results from \cite{universality} and \cite{polygeom}.

The equivalence between (b) and (c) above has a long history: see Remark~\ref{rmk:gen-tor}.

\subsection{Fiber functors}

The categories (a)--(e) above are $\bC$-linear tensor categories. However, only in (e) are the objects $\bC$-vector spaces (with extra structure), with the tensor product being the usual one on the underlying vector space. One can therefore view the equivalence $\Mod_K \cong \Rep^{\alg}(\Gamma_{\omega})$ as a fiber functor on $\Mod_K$. We thus get one such fiber functor for each choice of $\omega$. We show (\S \ref{s:fiber}) that all fibers functors are obtained in essentially this manner.

\subsection{Relation to previous work}

This paper is closely related to four threads of recent work:
\begin{itemize}
\item The papers \cite{bde,universality,polygeom,draisma,des} develop aspects of infinite dimensional $\GL$-equivariant algebraic geometry. These theories are based on $\GL$-algebras, which is the main connection to this paper. A few key arguments in this paper are in fact modeled on those from \cite{polygeom}. The work of Kazhdan--Ziegler \cite{kaz1,kaz2,kaz3,kaz4} is closely related.
\item The papers \cite{sym2noeth,periplectic,isomeric,symc1,symu1,symc1sp} study the module theory of a handful of specific $\GL$-algebras (and similar objects). The results of this paper generalize many of the results from those papers.
\item The papers \cite{koszulcategory,grantcharov,penkovserganova,penkovstyrkas,serganova,infrank} study the stable representation theory of classical (super)groups. The results summarized in \S \ref{ss:repgs} are all analogs of results from these papers (especially \cite{infrank}).
\item The papers \cite{semilin,rovinsky,rovinsky2,rovinsky3} study the semi-linear representation theory of the infinite symmetric group, which is thematically similar to much of the work in this paper.
\end{itemize}

\subsection{Further work}

In this paper, we give a fairly complete description of $\Mod_K$ when $K$ is a rational $\GL$-field. While we do prove some results for more general $\GL$-fields (see Theorem~\ref{thm:gen-struc}), there is still much left to be done in this direction. We hope to treat this in a future paper.

In the study of modules over $\GL$-algebras, it is also important to understand the generic categories $\Mod_R^{\gen}$ when $R$ is a ``$\GL$-domain'' (this means $\fa \fb=0$ implies $\fa=0$ or $\fb=0$ when $\fa$ and $\fb$ are $\GL$-ideals, which is a weaker condition than being a domain). In \cite{tcaprimes}, we gave a useful way of understanding the $\GL$-domain condition in terms of super mathematics, and we believe this should allow us to say something about these generic categories. We hope to return to this topic too.

\subsection{Open questions}

We list a few questions or problems raised by this work:
\begin{enumerate}
\item How much of standard Lie theory can be carried over to the generalized stabilizers $\Gamma_{\omega}$? Is there a Dynkin diagram, Cartan matrix, Weyl group, etc.?
\item Is there a Tannakian perspective that allows one to recover the generalized stabilizer $\Gamma_{\omega}$ from the fiber functor $\Phi_{\omega} \colon \Mod_K \to \Vec_k$?
\item Prove Theorem~\ref{thm:genstab} for general $K$.
\item What are the derived specializations of simple objects of $\Mod_K$? (See Remark~\ref{rmk:special}.)
\item We introduce the concept of ``germinal subgroup'' to define generalized stabilizers. While our definitions work for the purposes of this paper, we are not sure if they are optimal. For instance, our conditions do not say anything about inverses. It would be good to have more clarity on this point.
\item In this paper, we consider generalized stabilizers for actions of $\GL$ on infinite dimensional varieties. Are there other situations where generalized stabilizers are interesting? For example, one could consider generalized stabilizers arising from actions of the infinite symmetry group on infinite dimensional varieties.
\end{enumerate}

\subsection{Outline}

In \S \ref{s:bg} we provide background about $\GL$-algebras and related concepts.  In \S \ref{s:shift} we prove our two main technical theorems on $K$-modules, the shift and embedding theorems. We apply these results in \S \ref{s:struc} to deduce our main structural results on semi-linear representations. These results are in turn used in \S \ref{s:brauer} to obtain the connection to an analog of the Brauer category, which yields an analog of Weyl's construction and a universal property for $\Mod_K$. In \S \ref{s:fiber}, we classify the fiber functors of $\Mod_K$. In  \S \ref{s:genstab} we introduce germinal subgroups and generalized stabilizers in the abstract. Finally, in \S \ref{s:glstab}, we apply these concepts to $\GL$-varieties.

\subsection*{Acknowledgments}

We thank Arthur Bik, Jan Draisma, Rob Eggermont, Nate Harman, Steven Sam, and David Treumann for helpful conversations. In particular, Proposition~\ref{prop:Pi-Zariski} came from an e-mail exchange with Bik, Draisma, and Eggermont, and the material in \S \ref{ss:catthy} came from unpublished notes with Sam.

\section{\texorpdfstring{$\GL$}{GL}-equivariant algebra and geometry} \label{s:bg}

In this section, we review background material on polynomial representations, $\GL$-algebras, $\GL$-varieties, and related concepts. Additional details on these topics can be found in \cite{expos} and \cite{polygeom}.

\subsection{Polynomial representations}

Fix, for the entirety of the paper, a field $k$ of characteristic~0. Put $\GL=\bigcup_{n \ge 1} \GL_n(k)$, regarded as a discrete group. We let $\bV=\bigcup_{n \ge 1} k^n$ be the standard representation of $\GL$. We say that a representation of $\GL$ on a $k$-vector space is \emph{polynomial} if it appears as a subquotient of a (possibly infinite) direct sum of tensor powers of $\bV$. We let $\Rep^{\pol}(\GL)$ denote the category of polynomial representations. It is a semi-simple Grothendieck abelian category that is closed under tensor product.

For a partition $\lambda$, we let $\bS_{\lambda}$ denote the corresponding Schur functor. The simple polynomial representations are exactly those of the form $\bS_{\lambda}(\bV)$. Thus every polynomial representation decomposes as a (perhaps infinite) direct sum of $\bS_{\lambda}(\bV)$'s.

Every polynomial representation of $\GL$ carries a natural grading, with $\bS_{\lambda}(\bV)$ concentrated in degree $\vert \lambda \vert$, the size of the partition $\lambda$. This grading is compatible with tensor products: $\bS_{\lambda}(\bV) \otimes \bS_{\mu}(\bV)$ is concentrated in degree $\vert \lambda \vert+\vert \mu \vert$. The degree~0 piece of a polynomial representation $V$ is exactly the invariant subspace $V^{\GL}$.

We now introduce some non-standard notation that will be convenient for working with these objects. We write $k^{\oplus \lambda}$ in place of $\bS_{\lambda}(\bV)$. More generally, for a $k$-vector space $V$ we put $V^{\oplus \lambda}=V \otimes_k k^{\oplus \lambda}$; note that if $R$ is a $k$-algebra then $R^{\oplus \lambda}$ is naturally a free $R$-module. A \emph{tuple of partitions} (often simply called a \emph{tuple}) is a tuple $\ulambda=[\lambda_1, \ldots, \lambda_r]$, where each $\lambda_i$ is a partition. We put $k^{\oplus \ulambda}=\bigoplus_{i=1}^r k^{\oplus \lambda_i}$, and define $V^{\oplus \ulambda}$ similarly. We say that $\ulambda$ is \emph{pure} if it does not contain the empty partition. (This terminology comes from \cite{polygeom}.)

The category of polynomial representations is equivalent to the category of polynomial functors, with the representation $\bS_{\lambda}(\bV)$ corresponding to the functor $\bS_{\lambda}$. Given a polynomial representation $V$ and a vector space $U$, we let $V\{U\}$ be the result of regarding $V$ as a polynomial functor and evaluating on $U$. In the important special case where $U=k^n$, we can identify $V\{U\}$ with the invariant space $V^{G(n)}$, where $G(n)$ is defined in \S \ref{ss:shift}. For example, if $V=k^{\oplus \lambda}$ then $V\{k^n\}=\bS_{\lambda}(k^n)$.

\subsection{The maximal polynomial subrepresentation}

Suppose that $V$ is an arbitrary $k$-linear representation of $\GL$. We say that an element $x \in V$ is \defn{polynomial} if the subrepresentation it generates is a polynomial representation. We let $V^{\pol}$ be the set of all polynomial elements in $V$. It can be characterized as the maximal polynomial subrepresentation of $V$. Moreover, if $\REP(\GL)$ denotes the category of all $k$-linear representations of $\GL$ then $V \mapsto V^{\pol}$ is the right adjoint of the inclusion functor $\Rep^{\pol}(\GL) \to \REP(\GL)$. As such, $(-)^{\pol}$ is left-exact and continuous; it is not exact.

\subsection{The shift operation} \label{ss:shift}

Recall that $G(n)$ is the subgroup of $\GL$ consisting of block matrices of the form
\begin{displaymath}
\begin{pmatrix} 1 & 0 \\ 0 & \ast \end{pmatrix},
\end{displaymath}
where the top left block has size $n \times n$. We have a group isomorphism
\begin{displaymath}
\GL \to G(n), \qquad A \mapsto \begin{pmatrix} 1 & 0 \\ 0 & A \end{pmatrix}.
\end{displaymath}
Given some kind of object $X$ equipped with an action of $\GL$, we define its $n$th \defn{shift}, denoted $\Sh_n(X)$, to be the same object $X$ but with $\GL$ acting through the self-embedding $\GL \cong G(n) \subset \GL$.

One easily sees that if $V$ is a polynomial representation of $\GL$ then $\Sh_n(V)$ is also such a representation. From the polynomial functor point of view, we have
\begin{displaymath}
(\Sh_n{V})\{U\}=V\{k^n \oplus U\}.
\end{displaymath}
If $V$ has finite length then so does $\Sh_n(V)$. It follows that if $\ulambda$ is a tuple then there is another tuple, which we denote by $\sh_n(\ulambda)$, such that $\Sh_n(k^{\oplus \ulambda})=k^{\oplus \sh_n(\ulambda)}$. If $\ulambda=[\lambda]$ consists of a single partition, we write $\sh_n(\lambda)$ in place of $\sh_n(\ulambda)$. In this case, $\sh_n(\lambda)$ contains $\lambda$ exactly once, and all other partitions in it are strictly smaller.

\subsection{\texorpdfstring{$\GL$}{GL}-algebras}

A \emph{$\GL$-algebra} (over $k$) is a commutative algebra object in the tensor category $\Rep^{\pol}(\GL)$; thus, it is a commutative (and associative and unital) $k$-algebra equipped with an action of the group $\GL$ by algebra automorphisms, under which it forms a polynomial representation. Let $R$ be a $\GL$-algebra. By an \emph{$R$-module} we mean a module object in $\Rep^{\pol}(\GL$). Explicitly, this is an ordinary $R$-module $M$ equipped with a compatible action of $\GL$ under which $M$ forms a polynomial representation. We let $\Mod_R$ denote the category of modules, which is easily seen to be  a Grothendieck abelian category.

We say that $R$ is \emph{$\GL$-generated} (over $k$) by a set of elements if $R$ is generated as a $k$-algebra by the orbits of these elements. We say that $R$ is \emph{finitely $\GL$-generated} if it is $\GL$-generated by a finite set. We similarly speak of $\GL$-generation for $R$-modules.

We say that a $\GL$-algebra is \emph{integral} if it is integral in the usual sense (i.e., it is a domain). We will require the following important shift theorem from \cite{polygeom}.

\begin{theorem} \label{thm:bdes-shift}
Let $R$ be an integral $\GL$-algebra that is finitely $\GL$-generated. Then there exists $n \ge 0$, a non-zero $\GL$-invariant element $f \in \Sh_n(R)$, and an isomorphism $\Sh_n(R)[1/f] \cong A \otimes \Sym(k^{\oplus \usigma})$ for some finitely generated integral $k$-algebra $A$ (with trivial $\GL$-action) and pure tuple $\usigma$.
\end{theorem}

\begin{proof}
This is \cite[Theorem~5.1]{polygeom}, phrased in terms of coordinate rings.
\end{proof}

\subsection{\texorpdfstring{$\GL$}{GL}-varieties} \label{ss:glvar}

An \emph{affine $\GL$-scheme} is an affine scheme $X$ over $k$ equipped with an action of the discrete group $\GL$ such that $\Gamma(X, \cO_X)$ forms a polynomial representation of $\GL$. Every affine $\GL$-scheme has the form $\Spec(R)$ where $R$ is a $\GL$-algebra. An \emph{affine $\GL$-variety} is a reduced affine $\GL$-scheme $X$ such that $\Gamma(X, \cO_X)$ is finitely $\GL$-generated over $k$.

For a tuple $\ulambda$, let $\bA^{\ulambda}$ be the spectrum of the ring $\Sym(k^{\oplus \ulambda})$. This is an affine $\GL$-variety. Moreover, every affine $\GL$-variety is isomorphic to a closed $\GL$-subvariety of some $\bA^{\ulambda}$. Thus, in the theory of $\GL$-varieties, the $\bA^{\ulambda}$ play the same role as the ordinary affine spaces $\bA^n$ in ordinary algebraic geometry.

Let $X$ be an affine $\GL$-variety and let $x$ be a (scheme-theoretic) point of $X$. We let $\ol{O}_x$ be the Zariski closure of the orbit $\GL \cdot x$ of $x$ (see \cite[\S 3.1]{polygeom}). We say that $x$ is \emph{$\GL$-generic} if $\ol{O}_x=X$. Such points play a similar role to generic points in ordinary algebraic geometry. We define the \emph{generalized orbit} of $x$, denoted $O_x$, to be the set of all points $y$ such that $\ol{O}_x=\ol{O}_y$ (see \cite[\S 3.2]{polygeom}).

Write $X=\Spec(R)$ where $R$ is a $\GL$-algebra. Recall that for a vector space $U$ we let $R\{U\}$ be the result of treating $R$ as a polynomial functor and evaluating on $U$; this is a $k$-algebra equipped with an action of $\GL(U)$. We put $X\{U\}=\Spec(R\{U\})$. The standard inclusion $k^n \to \bV$ induces a ring homomorphism $R\{k^n\} \to R$, and thus a map of $k$-schemes $\pi_n \colon X \to X\{k^n\}$. Since $R$ is the union of the $R\{k^n\}$, it follows that $X$ is the inverse limit of the $X\{k^n\}$. We define the \emph{$\Pi$-topology} on $X$ to be the inverse limit topology, where each $X\{k^n\}$ is given the discrete topology. The $\Pi$-topology is actually quite concrete: if $k$ is algebraically closed then the set of closed points of $\bA^{\ulambda}$ is identified with a product of $k$'s, and the $\Pi$-topology is just the usual product topology; thus a sequence of $k$-points of $\bA^{\ulambda}$ converges if each coordinate is eventually constant. One easily sees that any Zariski closed set is $\Pi$-closed (see \cite[Proposition~2.3]{svar}).

We require the following result that relates the Zariski and $\Pi$-topologies:

\begin{proposition} \label{prop:Pi-Zariski}
Suppose that $k$ is algebraically closed. Let $X$ be a $\GL$-variety and let $x$ and $y$ be $k$-points of $X$. Then the following conditions are equivalent:
\begin{enumerate}
\item The orbits $\GL \cdot x$ and $\GL \cdot y$ have the same Zariski cloure.
\item The orbits $\GL \cdot x$ and $\GL \cdot y$ have the same $\Pi$-closure.
\end{enumerate}
\end{proposition}

\begin{proof}
Suppose (b) holds. Then $y$ belongs to the $\Pi$-closure of $\GL \cdot x$, which is contained in the Zariski closure of $\GL \cdot x$. We thus see that $\GL \cdot y$ is contained in the Zariski closure of $\GL \cdot x$, and so the Zariski closure of $\GL \cdot y$ is contained in the Zariski closure of $\GL \cdot x$. By symmetry, the reverse inclusion holds as well, which yields (a).

Now suppose that (a) holds. We may as well replace $X$ with the Zariski closure of $\GL \cdot x$, and so that $x$ and $y$ are $\GL$-generic in $X$. Let $\phi \colon B \times \bA^{\ulambda} \to X$ be a typical morphism (see \cite[\S 8.1]{polygeom}), where $B$ is an irreducible variety and $\ulambda$ is a pure tuple. Let $(b,\tilde{x}) \in B \times \bA^{\ulambda}$ be a $k$-point lifting $x$, which exists by \cite[Proposition~7.15]{polygeom}, and let $Z$ be the closure of the $\GL$-orbit of $(b,\tilde{x})$. Then $\phi \vert_Z$ is dominant since its image contains $x$, and so, by the definition of typical, $Z=B \times \bA^{\ulambda}$. It follows that $B=\{b\}$ is a point and $\tilde{x}$ is $\GL$-generic in $\bA^{\ulambda}$. In what follows we ignore $B$, and regard $\phi$ as a morphism $\phi \colon \bA^{\ulambda} \to X$ satisfying $\phi(\tilde{x})=x$.

The image of $\phi$ contains a non-empty open subset of $X$ by \cite[Theorem~7.13]{polygeom}. Since $y$ belongs to every non-empty $\GL$-subset of $X$ \cite[Proposition~3.4]{polygeom}, we see that $y \in \im(\phi)$. Thus, applying \cite[Proposition~7.15]{polygeom} again, we can find a $k$-point $\tilde{y}$ of $\bA^{\ulambda}$ such that $\phi(\tilde{y})=y$.

Let $\pi_n \colon \bA^{\ulambda} \to \bA^{\ulambda}\{K^n\}$ be the natural map. By \cite[Corollary 2.6.3]{universality}, the restriction of $\pi_n$ to $\GL \cdot \tilde{x}$ is surjective on $k$-points. We can thus find $g_n \in \GL$ such that $\pi_n(g_n\tilde{x})=\tilde{y}$. We therefore see that the sequence $(g_n \tilde{x})_{n \ge 1}$ converges to $\tilde{y}$ in the $\Pi$-topology. Since $\phi$ is $\Pi$-continuous, it follows that the sequence $(g_n x)_{n \ge 1}$ converges to $y$ in the $\Pi$-topology. Thus $y$, and therefore $\GL \cdot y$, and therefore the $\Pi$-closure of $\GL \cdot y$, is contained in the $\Pi$-closure of $\GL \cdot x$. The reverse inclusion follows by symmetry, and so (b) holds.
\end{proof}

\begin{remark} \label{rmk:gen-orb}
Proposition~\ref{prop:Pi-Zariski} shows that, when working with closed points over an algebraically closed field, one can define the generalized orbit of $x$ using the $\Pi$-topology (as we did in \S \ref{ss:intro-gen-orb}): that is, a $k$-point $y$ belongs to $O_x$ if and only if one can find sequences $(g_n)_{n \ge 1}$ and $(h_n)_{n \ge 1}$ in $\GL$ such that $g_n x \to y$ and $h_n y \to x$ in the $\Pi$-topology.
\end{remark}

\begin{remark}
We only apply Proposition~\ref{prop:Pi-Zariski} when $X=\bA^{\ulambda}$, in which case the proof simplifies some. However, we feel that the general statement is important enough that it is worth recording here.
\end{remark}

\subsection{\texorpdfstring{$\GL$}{GL}-fields}

A \defn{$\GL$-field} over $k$ is a field extension $K/k$ equipped with an action of $\GL$ by $k$-automorphisms such that every element of $K$ can be expressed in the form $a/b$ with $a,b \in K^{\pol}$. If $K$ is a $\GL$-field then $K^{\pol}$ is an integral $\GL$-algebra over $k$, and $K=\Frac(K^{\pol})$. Thus every $\GL$-field can be realized as the fraction field of an integral $\GL$-algebra.

Let $K$ be a $\GL$-field. A \defn{$K$-module} is a semi-linear representation $M$ of $\GL$ over $K$ such that every element of $M$ has the form $ax$ with $a \in K$ and $x \in M^{\pol}$. One easily sees that the category $\Mod_K$ of $K$-modules is an abelian category satisfying the (AB5) condition. Moreover, if $M$ is any $K$-module then there is a surjection $K \otimes V \to M$ for some polynomial representation $V$ (take $V=M^{\pol}$), which shows that the objects $K^{\oplus \lambda}$ form a generating set; thus $\Mod_K$ is a Grothendieck abelian category.

We say that $K$ is \defn{finitely $\GL$-generated} over $k$ if it is generated as a field extension by the $\GL$-orbits of finitely many elements. We say that $K$ is \defn{rational} over $k$ if it has the form $\Frac(\Sym(k^{\oplus \usigma}))$ for some tuple $\usigma$. The \defn{invariant subfield} of $K$, denoted $K^{\GL}$, is the subfield of $K$ consisting of all elements that are invariant under $\GL$. It is an extension of $k$. If $K$ is finitely $\GL$-generated over $k$ then $K^{\GL}$ is finitely generated over $k$ (\cite[Proposition~5.8]{polygeom}).

\begin{proposition} \label{prop:unirat-K}
Let $K$ be a $\GL$-field that is finitely $\GL$-generated over $k$. Then there exists $n \ge 0$ such that $\Sh_n(K)$ is rational over its invariant subfield.
\end{proposition}

\begin{proof}
One easily sees that $K$ can be $\GL$-generated by finitely many polynomial elements. We can thus find a finitely $\GL$-generated $k$-subalgebra $R$ of $K$ such that $K=\Frac(R)$. Apply Theorem~\ref{thm:bdes-shift} to write $\Sh_n(R)[1/f] \cong A \otimes \Sym(k^{\oplus \usigma})$ where $A$ is a $k$-algebra with trivial $\GL$-action and $\usigma$ is a pure tuple. Taking fraction fields, we find $\Sh_n(K) \cong \Frac(\Sym(\ell^{\oplus \usigma}))$ where $\ell=\Frac(A)$. It follows from \cite[Proposition~5.7]{polygeom} that $K^{\GL} \cong \ell$, and so $\Sh_n(K)$ is rational over its invariant subfield.
\end{proof}

\subsection{Generic categories} \label{ss:gen}

Let $R$ be an integral $\GL$-algebra. We say that an $R$-module $M$ is \defn{torsion} if every element of $M$ is annihilated by a non-zero element of $R$. The category $\Mod_R^{\tors}$ of torsion $R$-modules is a Serre subcategory of $\Mod_R$. We define the \defn{generic category} of $R$, denoted $\Mod_R^{\gen}$, to be the Serre quotient $\Mod_R/\Mod_R^{\tors}$.

The generic category can be described in terms of semi-linear representations. Let $K=\Frac(R)$. We have a functor
\begin{displaymath}
T \colon \Mod_R \to \Mod_K, \qquad T(M) = K \otimes_R M.
\end{displaymath}
We also have a functor
\begin{displaymath}
S \colon \Mod_K \to \Mod_R, \qquad T(N) = N^{\pol}.
\end{displaymath}
Indeed, if $N$ is a $K$-module then $R \otimes N^{\pol}$ is a polynomial representation, so its image under the natural map $R \otimes N^{\pol} \to N$ consists of polynomial elements, and is therefore contained in $N^{\pol}$; this shows that $N^{\pol}$ is stable under multiplication by $R$, and is thus an $R$-module.

\begin{proposition} \label{prop:genK}
We have the following:
\begin{enumerate}
\item The functor $T$ is exact and kills torsion modules. The induced functor $\Mod_R^{\gen} \to \Mod_K$ is an equivalence.
\item The functors $(T,S)$ form an adjoint pair.
\item The co-unit $TS \to \id$ is an isomorphism.
\end{enumerate}
\end{proposition}

\begin{proof}
See \cite[\S 2.4]{sym2noeth}.
\end{proof}

We say that an $R$-module $M$ is \defn{saturated} if the natural map $M \to S(T(M))$ is an isomorphism. We will require the following result concerning this concept:

\begin{proposition} \label{prop:sat}
Let $\usigma$ be a pure tuple, let $R=\Sym(k^{\oplus \usigma})$, and let $V$ be a polynomial representation. Then $R \otimes V$ is a saturated $R$-module.
\end{proposition}

\begin{proof}
See \cite[Proposition~2.8]{sym2noeth}.
\end{proof}

\section{The shift and embedding theorems} \label{s:shift}

In this section, we prove our two main technical results on $K$-modules: the shift theorem (Theorem~\ref{thm:shift}) and the embedding theorem (Theorem~\ref{thm:embed}).

\subsection{A preliminary result}

The following proposition is the key input needed for the shift theorem proven in the subsequent subsection. It is a linear analog of \cite[Theorem~4.2]{polygeom}, a result that was essentially taken from arguments in \cite{draisma}.

\begin{proposition} \label{prop:reduction}
Let $R$ be an integral $\GL$-algebra, let $\lambda$ be a partition, let $F$ and $M$ be $R$-modules, and suppose we have a surjection of $R$-modules
\begin{displaymath}
R^{\oplus \lambda} \oplus F \to M.
\end{displaymath}
Then at least one of the following holds:
\begin{enumerate}
\item The given map induces an isomorphism $R^{\oplus \lambda} \oplus N \to M$, where $N$ is a quotient of $F$.
\item There exists $n \ge 0$ and a non-zero $\GL$-invariant element $f \in \Sh_n(R)$ such that the natural map
\begin{displaymath}
\Sh_n(R)[1/f]^{\oplus \umu} \oplus \Sh_n(F)[1/f] \to \Sh_n(M)[1/f]
\end{displaymath}
is surjective, where $\umu$ is obtained from $\sh_n(\lambda)$ by deleting $\lambda$.
\end{enumerate}
\end{proposition}

We require some preparation before giving the proof. A \emph{weight} of $\GL$ is a tuple $\lambda=(\lambda_1, \lambda_2, \ldots)$ where $\lambda_i \in \bZ$ for all $i$ and $\lambda_i=0$ for $i \gg 0$. For a finite subset $A$ of $[\infty]=\{1,2,\ldots\}$, we let $1^A$ be the weight that is~1 at the coordinates in $A$, and~0 away from $A$. We also write $1^n$ in place of $1^A$ when $A=[n]$.

Suppose that $V$ is a polynomial representation and $\lambda$ is a weight. We say that $v \in V$ is a \emph{weight vector} of weight $\lambda$ if whenever $g=\diag(a_1,a_2,\ldots)$ we have
\begin{displaymath}
gv=\big( \prod_{i \ge 1} a_i^{\lambda_i} \big) \cdot v.
\end{displaymath}
We let $V_{\lambda}$ be the space of all weight vectors of weight $\lambda$; this is the \emph{$\lambda$ weight space}. The space $V$ is the direct sum of its weight spaces $V_{\lambda}$ over all $\lambda$. Moreover, if $V_{\lambda}$ is non-zero then $\lambda$ is non-negative in the sense that $\lambda_i \ge 0$ for all $i$.

The weight space $V_{1^n}$ carries a representation of $\fS_n \subset \GL$. Let $\Rep^{\pol,n}(\GL)$ be the subcategory of $\Rep^{\pol}(\GL)$ spanned by representations of degree $n$. One formulation of Schur--Weyl duality states that the functor
\begin{align*}
\Rep^{\pol,n}(\GL) &\to \Rep(\fS_n) \\
V &\mapsto V_{1^n}
\end{align*}
is an equivalence of categories.

\begin{lemma} \label{lem:reduction}
Let $V$ and $W$ be polynomial representations of degrees $n$ and $m$, with $V$ irreducible, let $S$ be a subset of $[n+m]$ of cardinality $n$, and let $U$ be a non-zero subrepresentation of $V \otimes W$. Then $U$ contains a vector of the form $x=\sum_{i=1}^r v_i \otimes w_i$, for some $r \ge 1$, such that:
\begin{itemize}
\item $v_i$ is a weight vector of $V$ of weight $1^{A_i}$ and $w_i$ is a weight vector of $W$ of weight $1^{B_i}$, where $A_i$ and $B_i$ are disjoint and $A_i \cup B_i=[n+m]$;
\item we have $A_1=S$, and $v_1$ and $w_1$ are non-zero;
\item we have $A_i \ne S$ for $i>1$.
\end{itemize}
\end{lemma}

\begin{proof}
We may as well assume $S=[n]$. By Schur--Weyl duality, the $1^{n+m}$-weight space of $U$ is non-zero. We can thus find a non-zero element $x$ of $U$ of the form $x=\sum_{i=1}^r v_i \otimes w_i$ satisfying the first condition, and with the $v_i$ and $w_i$ linearly independent. Applying an element of the symmetric group $\fS_{n+m} \subset \GL$, we can assume that $v_1$ has weight $1^n$. Relabeling, we can assume that $v_1, \ldots, v_k$ have weight $1^n$, and that the remaining $v_i$ have weight $\ne 1^n$.

Now, by Schur--Weyl duality, the $1^n$ weight space of $V$ is an irreducible representation of $\fS_n$ (acting through the standard inclusion $\fS_n \subset \GL$). Since $v_1, \ldots, v_k$ are linearly independent elements, we can find $a \in \bC[\fS_n]$ such that $av_1=v_1$ and $av_i=0$ for $2 \le i \le k$, Since $w_1$ has weight $1^{B_1}$ with $B_1=\{n+1, \ldots, n+m\}$, the group $\fS_n$ acts trivially on it, and so $a(v_1 \otimes w_1)=v_1 \otimes w_1$. For $k<i$ the element $av_i$ is a sum of weight vectors having weight of the form $1^A$ with $A \ne [n]$. We thus see that $ax$ is an element of $U$ of the required form.
\end{proof}

\begin{proof}[Proof of Proposition~\ref{prop:reduction}]
Let $K$ be the kernel of $R^{\oplus \lambda} \oplus F \to M$, and let $\ol{K}$ be the projection of $K$ to $R^{\oplus \lambda}$. If $\ol{K}=0$ then $K$ is contained in $F$, and case (a) holds with $N=F/K$. Suppose now that $\ol{K} \ne 0$. Let $n=\vert \lambda \vert$ and let $m \ge 0$ be such that $\ol{K}$ has a non-zero element of degree $n+m$. Recall that $R^{\oplus \lambda}=E \otimes R$, where $E=k^{\oplus \lambda}$. Applying Lemma~\ref{lem:reduction}, we can find an element $x$ of $\ol{K}$ of the form $x=\sum_{i=1}^r f_i e_i$, where:
\begin{itemize}
\item $e_i$ is a weight vector of $E$ of weight $1^{A_i}$ and $f_i$ is a weight vector of $R$ of weight $1^{B_i}$, where $A_i$ and $B_i$ are disjoint and $A_i \cup B_i=[n+m]$;
\item $A_1=\{m+1, \ldots, n+m\}$, and $f_1$ and $e_1$ are non-zero.
\item $A_i\ne \{m+1,\ldots,n+m\}$ for $i>1$.
\end{itemize}
Say that a weight $\lambda$ is \emph{big} if $\lambda_i=0$ for $i \in [m]$, and \emph{small} otherwise. Let $E^{\rm big}$ and $E^{\rm small}$ be the sum of the big and small weight spaces in $E$. Then we have a decomposition of $G(m)$-representations
\begin{displaymath}
E=E^{\rm big} \oplus E^{\rm small}.
\end{displaymath}
Identifying $G(m)$ with $\GL$, this becomes the decomposition
\begin{displaymath}
\Sh_m(E) = k^{\oplus \lambda} \oplus k^{\oplus \umu}.
\end{displaymath}
We thus see that $E^{\rm big}$ is irreducible as a $G(m)$-representation. Note that $e_1$ is a non-zero element of $E^{\rm big}$ (and thus generates it as a $G(m)$-representation), and that $f_1$ is $G(m)$-invariant (as it has weight $1^m$).

Let $y \in F$ be such that $x+y \in K$. Let $M'$ be the image of $(E^{\rm small} \otimes R) \oplus F$ in $M$. Since $x+y$ maps to~0 in $M$, we see that the image of $f_1 e_1$ in $M$ belongs to $M'$, and so the image of $e_1$ belongs to $M'[1/f_1]$. Since $M'$ is $G(m)$-stable, $f_1$ is $G(m)$-invariant, and $e_1$ generates $E^{\rm big}$ as a $G(m)$-representation, we see that any element of $E^{\rm big} \otimes R$ maps into $M'[1/f_1]$. Thus $M'[1/f_1]=M[1/f_1]$, and the result follows.
\end{proof}

\subsection{The shift theorem}

We now prove the first main result of this section. It is an analog of \cite[Theorem~5.1]{polygeom}.

\begin{theorem}[Shift theorem] \label{thm:shift}
Let $R$ be an integral $\GL$-algebra and let $M$ be a finitely generated $R$-module. Then there exists $n \ge 0$, a tuple $\ulambda$, and a non-zero $\GL$-invariant element $f \in \Sh_n(R)$ such that we have an isomorphism $\Sh_n(M)[1/f] \cong \Sh_n(R)[1/f]^{\oplus \ulambda}$ of $\Sh_n(R)[1/f]$-modules.
\end{theorem}

\begin{proof}
Say that an $R$-module is \defn{good} if the conclusion of the theorem holds for it. Consider the following statement, for a tuple $\umu$:
\begin{itemize}
\item[$S(\umu)$] If $R$ is an integral $\GL$-algebra and $M$ is a quotient module of $R^{\oplus \umu}$ then $M$ is good.
\end{itemize}
It suffices to prove $S(\umu)$ for all tuples $\umu$. The \emph{magnitude} of a tuple $\umu$, denoted $\magn(\umu)$, is the tuple $(n_0, n_1, \ldots)$ where $n_i$ is the number of partitions of size $i$ in $\mu$. We order magnitudes lexicographically; this is a well-order. We can thus prove $S(\umu)$ by induction on $\magn(\umu)$. Thus let $\umu$ be given, and suppose $S(\unu)$ holds for all $\unu$ with $\magn(\unu)<\magn(\umu)$. We prove $S(\umu)$. If $\umu$ is empty the statement is vacuous, so suppose this is not the case.

Let $R$ be an integral $\GL$-algebra and let $M$ be a quotient of $R^{\oplus \umu}$. Let $\kappa$ be a partition in $\umu$ of maximal size, and let $\unu$ be the tuple obtained from $\umu$ by deleting $\kappa$. We thus have a surjection $R^{\oplus \kappa} \oplus R^{\oplus \unu} \to M$. We apply Proposition~\ref{prop:reduction} with $F=R^{\oplus \unu}$. We consider the two cases separately.

Suppose case (a) holds. Then $M=R^{\oplus \kappa} \oplus N$ where $N$ is a quotient of $R^{\oplus \unu}$. Since $\unu$ has smaller magnitude than $\umu$, statement $S(\unu)$ holds, and so $N$ is good. It is clear then that $M$ is good as well.

Now suppose case (b) holds. Then there is some $n \ge 0$ and a $\GL$-invariant function $f \in \Sh_n(R)$ such that the natural map
\begin{displaymath}
\Sh_n(R)[1/f]^{\oplus \ul{\rho}} \oplus \Sh_n(R^{\oplus \unu})[1/f] \to \Sh_n(M)[1/f]
\end{displaymath}
is a surjection, where $\sh_n(\kappa)=[\kappa] \cup \ul{\rho}$. Now, the left side above has the form $\Sh_n(R)[1/f]^{\oplus \ul{\sigma}}$ where $\ul{\sigma}=\ul{\rho} \cup \sh_n(\unu)$. The tuple $\ul{\sigma}$ has smaller magnitude than $\umu$, and so statement $S(\ul{\sigma})$ holds. We thus see that $\Sh_n(M)[1/f]$ is good as a $\Sh_n(R)[1/f]$-module, from which it easily follows that $M$ is good as an $R$-module. This completes the proof.
\end{proof}

We also have the following statement, which appears to be slightly stronger, but in fact follows easily from the theorem:

\begin{corollary} \label{cor:shift}
Let $R$ be an integral $\GL$-algebra and let $M$ be a finitely generated $R$-module. Then there exists $n \ge 0$, a tuple $\ulambda$, and a non-zero $\GL$-invariant element $f \in \Sh_n(R)$ such that there is an injection $\Sh_n(R)^{\oplus \ulambda} \to \Sh_n(M)$ of $\Sh_n(R)$-modules with cokernel annihilated by $f$.
\end{corollary}

\begin{proof}
Let $n$, $\ulambda$, and $f$ be as in the shift theorem, so that we have an isomorphism of $\Sh_n(R)[1/f]$-modules $\Sh_n(R)[1/f]^{\oplus \ulambda} \to \Sh_n(M)[1/f]$. Let $M'$ be the image of $\Sh_n(M)$ in $\Sh_n(M)[1/f]$. Scaling our isomorphism by an appropriate power of $f$, we can assume that $\Sh_n(R)^{\oplus \ulambda}$ maps into $M'$. Since $\Sh_n(R)^{\oplus \ulambda}$ is projective, we can find a lift $\Sh_n(R)^{\oplus \ulambda} \to \Sh_n(M)$ of our map, which is necessarily injective. Since this map is an isomorphism after inverting $f$, every element in the cokernel is annihilated by a power of $f$. But the cokernel is finitely $\GL$-generated and $f$ is $\GL$-invariant, so there is some power of $f$ that annihilates the entire cokernel. Replace $f$ by this power.
\end{proof}

The shift theorem for $\GL$-algebras implies an analogous result for $\GL$-fields:

\begin{corollary}
Let $K$ be a $\GL$-field and let $M$ be a finitely generated $K$-module. Then there exists $n \ge 0$ and a tuple $\ulambda$ such that we have an isomorphism $\Sh_n(M) \cong \Sh_n(K)^{\oplus \ulambda}$ of $\Sh_n(K)$-modules.
\end{corollary}

\begin{proof}
Let $R=K^{\pol}$ and let $N \subset M^{\pol}$ be a finitely $\GL$-generated $R$-module that spans $M$ over $K$. By Theorem~\ref{thm:shift}, we have an isomorphism $\Sh_n(N)[1/f] \cong \Sh_n(R)[1/f]^{\oplus \ulambda}$ for some $n$, $f$, and $\ulambda$. Tensoring up to $\Sh_n(K)$, we obtain the stated result.
\end{proof}

\subsection{Some consequences}

We now give a few consequences of the shift theorem.

\begin{proposition} \label{prop:loc-free}
Let $R$ be an integral $\GL$-algebra and let $M$ be a finitely generated $R$-module. Then there exists a non-empty open $\GL$-stable subset $U$ of $\Spec(R)$ such that $M_{\fp}$ is free over $R_{\fp}$ for all $\fp \in U$.
\end{proposition}

\begin{proof}
The shift theorem shows that $\Sh_n(M)[1/f]$ is free as an $\Sh_n(R)[1/f]$-module, for some non-zero $f$. Since freeness does not depend on the $\GL$-actions, it follows that $M[1/f]$ is free as an $R[1/f]$-module. Thus $M_{\fp}$ is free over $R_{\fp}$ for all $\fp \in D(f)$, where $D(f) \subset \Spec(A)$ is the distinguished open defined by $f$. Since the free locus is obvious $\GL$-stable, we can take $U=\bigcup_{g \in \GL} gD(f)$.
\end{proof}

\begin{corollary} \label{cor:flat}
Let $R$ be an integral $\GL$-algebra, let $M$ be an $R$-module, and let $\fp$ be a $\GL$-generic prime of $R$. Then $M_{\fp}$ is flat over $R_{\fp}$.
\end{corollary}

\begin{proof}
First suppose that $M$ is finitely generated. By Proposition~\ref{prop:loc-free}, $M$ is flat at an non-empty $\GL$-stable open subset of $\Spec(R)$. Such a subset contains all $\GL$-generic points \cite[Proposition~3.4]{polygeom}. Thus $M$ is flat at $\fp$. In general, write $M=\varinjlim M_i$ with each $M_i$ finitely generated. Then $M_{\fp}=\varinjlim (M_i)_{\fp}$ is a direct limit of flat modules, and thus flat.
\end{proof}

The following result shows that finitely generated modules have ``bounded torsion'' in an appropriate sense:

\begin{proposition}
Let $R$ be an integral $\GL$-algebra, let $M$ be a finitely generated $R$-module, and let $M_{\rm tors}$ be the torsion submodule of $M$. Then there exists a non-zero $f \in R$ such that $f M_{\rm tors}=0$.
\end{proposition}

\begin{proof}
Applying Corollary~\ref{cor:shift}, let $i \colon \Sh_n(R)^{\oplus \ulambda} \to \Sh_n(M)$ be an injection of $\Sh_n(R)$-modules with cokernel annihilated by $f$. Since $\Sh_n(M_{\rm tors})$ is torsion, it cannot intersect $\im(i)$, and so it injects into $\coker(i)$. Since $\coker(i)$ is annihilated by $f$, so is $M_{\rm tors}$.
\end{proof}

\subsection{The embedding theorem}

We now prove the second main result of this section.

\begin{theorem}[Embedding theorem] \label{thm:embed}
Let $A$ be an integral $k$-algebra, let $\usigma$ be a tuple, and let $R=A \otimes \Sym(k^{\oplus \usigma})$. Let $M$ be a finitely $\GL$-generated torsion-free $R$-module. Then there is a tuple $\umu$ and an injection $M \to R^{\oplus \umu}$ of $R$-modules.
\end{theorem}

We require some discussion before giving the proof. If $V$ is a polynomial representation of $\GL$ then $V$ is identified with $V\{\bV\}$ and $\Sh_n(V)$ is identified with $V\{k^n \oplus \bV\}$. The natural inclusion $\bV \to k^n \oplus \bV$ thus induces a map $V \to \Sh_n(V)$, which is injective. If $R$ is a $\GL$-algebra then the map $R \to \Sh_n(R)$ is one of $\GL$-algebras, and if $M$ is an $R$-module then the map $M \to \Sh_n(M)$ is one of $R$-modules. We say that $R$ is \emph{shift-free} if for each $n$ the $R$-module $\Sh_n(R)$ has the form $E_n \otimes R$ for some polynomial representation $E_n$. Theorem~\ref{thm:embed} thus follows from the following two lemmas.

\begin{lemma} \label{lem:shift-free}
Let $A$ be a $k$-algebra, let $\usigma$ be a tuple, and let $R=A \otimes \Sym(k^{\oplus \usigma})$. Then $R$ is shift-free.
\end{lemma}

\begin{proof}
We have $\Sh_n(R)=A \otimes \Sym(\Sh_n(k^{\oplus \usigma}))$. Write $\Sh_n(k^{\oplus \usigma})=k^{\oplus \usigma} \oplus k^{\oplus \utau(n)}$ for some tuple $\utau(n)$, and let $E_n=\Sym(k^{\oplus \utau(n)})$. Then $\Sh_n(R) \cong E_n \otimes R$, as $\GL$-algebras, and, in particular, as $R$-modules. Thus $R$ is shift-free.
\end{proof}

\begin{lemma}
Let $R$ be an integral shift-free $\GL$-algebra and let $M$ be a finitely $\GL$-generated torsion-free $R$-module. Then there is a tuple $\umu$ and an injection $M \to R^{\oplus \umu}$ of $R$-modules.
\end{lemma}

\begin{proof}
Applying the shift theorem (Theorem~\ref{thm:shift}), we have an isomorphism $i \colon \Sh_n(M)[1/f] = \Sh_n(R)[1/f]^{\oplus \ulambda}$ for some $n$, $f$, and $\ulambda$. Since $M$ is torsion-free, the natural map $\Sh_n(M) \to \Sh_n(M)[1/f]$ is injective. Scaling $i$ by a power of $f$, we can assume it maps $\Sh_n(M)$ into $\Sh_n(R)^{\oplus \ulambda}$. Composing with the natural map $M \to \Sh_n(M)$, we obtain an injection of $R$-modules $j \colon M \to \Sh_n(R)^{\oplus \ulambda}$. As $R$-modules, we have $\Sh_n(R) \cong E_n \otimes R$ for some polynomial representation $E_n$. Thus we can identify the target of $j$ with $F \otimes R$ where $F=k^{\oplus \ulambda} \otimes E_n$. Since $M$ is finitely generated, the image of $j$ is contained in $F_0 \otimes R$ for some finite length subrepresentation $F_0$ of $F$. Writing $F_0 \cong k^{\oplus \umu}$ for some tuple $\umu$ yields the result.
\end{proof}

There is also an embedding theorem for rational $\GL$-fields:

\begin{corollary} \label{cor:embed}
Let $\usigma$ be a pure tuple, let $K=\Frac(\Sym(k^{\oplus \usigma}))$, and let $M$ be a finitely generated $K$-module. Then there exists an injection $M \to K^{\oplus \ulambda}$ of $K$-modules for some tuple $\ulambda$.
\end{corollary}

\begin{proof}
Let $R=\Sym(k^{\oplus \usigma})$ and let $M_0 \subset M^{\pol}$ be a finitely generated $R$-module with $M=K \otimes_R M_0$. Since $M_0$ is contained in $M$, it is torsion-free. By Theorem~\ref{thm:embed}, there is an injection of $R$-modules $M_0 \to R^{\oplus \ulambda}$ for some tuple $\ulambda$. Tensoring up to $K$ gives the stated result.
\end{proof}

\begin{remark}
Theorem~\ref{thm:embed} is a linear analog of \cite[Theorem~5.4]{polygeom}. That result is stated for $\GL$-varieties, but if formulated in terms of $\GL$-algebras it states that certain $\GL$-algebras can be embedded into polynomial $\GL$-algebras, which is analogous to embedding modules into free modules.
\end{remark}

\begin{remark}
We do not know if there are any examples of shift-free $\GL$-algebras besides the ones appearing in Lemma~\ref{lem:shift-free}.
\end{remark}

\section{The main structural results for semi-linear representations} \label{s:struc}

In this section, we prove many of the results stated in \S \ref{ss:repgs} in the $\Mod_K$ setting. These results essentially follow in a formal manner from the embedding theorem. In \S \ref{ss:catthy}, we give an axiomatization of the formal arguments. In \S \ref{ss:rat-struct}, we apply this axiomatization to prove the results on $\Mod_K$, when $K$ is a rational $\GL$-field. Finally, in \S \ref{ss:genfield}, we prove some results for more general $\GL$-fields.

\subsection{Some category theory} \label{ss:catthy}

Let $\cA$ be a $k$-linear Grothendieck abelian category, let $\{I_{\lambda}\}_{\lambda \in \Lambda}$ be a set of non-zero objects in $\cA$, and let $\vert \cdot \vert \colon \Lambda \to \bZ_{\ge 0}$ be a function.  Suppose that the following conditions hold:
\begin{itemize}
\item[(A1)] Every object of $\cA$ is the union of its finitely generated subobjects.
\item[(A2)] For any $n$, there are only finitely many $\lambda \in \Lambda$ with $\vert \lambda \vert \le n$.
\item[(A3)] The object $I_{\lambda}$ is finitely generated, for all $\lambda \in \Lambda$.
\item[(A4)] The space $\Hom(I_{\lambda}, I_{\mu})$ is finite dimensional over $k$ for all $\lambda,\mu \in \Lambda$.
\item[(A5)] The ring $\End(I_{\lambda})$ is a division ring for all $\lambda \in \Lambda$.
\item[(A6)] We have $\Hom(I_{\lambda}, I_{\mu}) \ne 0$ only if $\vert \mu \vert< \vert \lambda \vert$ or $\lambda=\mu$.
\item[(A7)] Let $\lambda \in \Lambda$ and let $J$ be a direct sum of objects of the form $I_{\mu}$ with $\vert \mu \vert<\vert \lambda \vert$. Then there is no injection $I_{\lambda} \to J$.
\item[(A8)] Every finitely generated object of $\cA$ injects into a finite direct sum of the $I_{\lambda}$'s.
\end{itemize}
We introduce one more piece of notation: for $\lambda \in \Lambda$, we let $L_{\lambda}$ be the intersection of the kernels of all maps $I_{\lambda} \to I_{\mu}$ with $\vert \mu \vert<\vert \lambda \vert$. The object $L_{\lambda}$ is non-zero by (A7).

\begin{proposition} \label{prop:Ilambda}
In the above situation, we have the following:
\begin{enumerate}
\item Every object of $\cA$ is locally of finite length.
\item The $I_{\lambda}$'s are exactly the indecomposable injectives of $\cA$.
\item Every finite length object of $\cA$ has finite injective dimension.
\item The object $L_{\lambda}$ is simple, and is equal to the socle of $I_{\lambda}$. Every simple object is isomorphic to a unique $L_{\lambda}$.
\item The simple object $L_{\lambda}$ occurs in $I_{\lambda}$ with multiplicity one; the remaining simple constituents of $I_{\lambda}$ have the form $L_{\mu}$ with $\vert \mu \vert<\vert \lambda \vert$.
\end{enumerate}
\end{proposition}

We break the proof up into a series of lemmas. We assume that $\cA$ satisfies (A1)--(A8) in the following.

\begin{lemma} \label{lem:Ilambda-1}
Let $I$ be a finitely generated object of $\cA$. Suppose that every injection $I \to M$, with $M$ finitely generated, splits. Then $I$ is injective.
\end{lemma}

\begin{proof}
Let $M$ be a finitely generated object of $\cA$, let $N$ be a subobject of $M$, and let $N \to I$ be a given morphism. Consider the map $I \to (I \oplus M)/N$, where $N$ is embedded diagonally, which is easily seen to be injective. Since $(I  \oplus M)/N$ is finitely generated, this map splits by hypothesis. This yields a map $M \to I$ extending the given map $N \to I$. A variant of Baer's criterion (see \stacks{079G}) now shows that $I$ is injective. (The key point here is that $\cA$ is generated by its finitely generated objects, due to (A1).)
\end{proof}

\begin{lemma} \label{lem:Ilambda-2}
The object $I_{\lambda}$ is an indecomposable injective, for all $\lambda \in \Lambda$.
\end{lemma}

\begin{proof}
Since $\End(I_{\lambda})$ has no non-trivial idempotents, it follows that $I_{\lambda}$ is indecomposable. It is clear that $I_{\lambda}$ is injective if $\vert \lambda \vert<0$, since this hypothesis is void. Assume now that $I_{\mu}$ is injective for all $\mu$ with $\vert \mu \vert<n$ and let $\lambda$ satisfy $\vert \lambda \vert=n$. Suppose we have an injection $f \colon I_{\lambda} \to M$, with $M$ finitely generated. Choose an injection $g \colon M \to I$, where $I$ is a finite direct sum of $I_{\mu}$'s, which is possible by (A8). Write $I=I_1 \oplus I_2 \oplus I_3$, where $I_1$ is a sum of $I_{\mu}$'s with $\vert \mu \vert<n$, $I_2$ is a sum of $I_{\lambda}$'s and $I_3$ is a sum of $I_{\mu}$'s with $\vert \mu \vert \ge n$ and $\mu \ne \lambda$. Let $p_i$ be the projection of $I$ onto $I_i$. Then $p_3gf=0$ by (A6) and $p_1gf$ is not injective by (A7). Since $gf$ is injective, it follows that $p_2gf$ is non-zero. Thus $p_2g$, followed by a further projection, provides a map $h \colon M \to I_{\lambda}$ such that $hf$ is non-zero. Since $\End(I_{\lambda})$ is a division algebra by (A5), we can find $h' \in \End(I_{\lambda})$ such that $h'hf=\id$, and so $f$ is split. It follows from Lemma~\ref{lem:Ilambda-1} that $I_{\lambda}$ is injective. The result now follows by induction.
\end{proof}

\begin{lemma} \label{lem:Ilambda-3}
The object $L_{\lambda}$ is simple, and is the socle of $I_{\lambda}$.
\end{lemma}

\begin{proof}
Consider the natural map
\begin{displaymath}
f \colon I_{\lambda} \to J, \qquad J = \bigoplus_{\vert \mu \vert<n} \Hom(I_{\lambda}, I_{\mu})^* \otimes I_{\mu}.
\end{displaymath}
This is the universal map from $I_{\lambda}$ to a sum of $I_{\mu}$'s with $\vert \mu \vert<\vert \lambda \vert$. Thus $L_{\lambda}=\ker(f)$.

Suppose $N$ is a non-zero subobject of $L_{\lambda}$. The object $I_{\lambda}/N$ is finitely generated by (A3). Thus, by (A8), we have an injection $I_{\lambda}/N \to I$ where $I$ is a finite sum of $I_{\mu}$'s. Any map $I_{\lambda}/N \to I_{\mu}$ with $\vert \mu \vert \ge \vert \lambda \vert$ and $\mu \ne \lambda$ is automatically zero by (A6); similarly, any map $I_{\lambda}/N \to I_{\lambda}$ is zero, since any non-zero map $I_{\lambda} \to I_{\lambda}$ is injective by (A5). It follows that $I$ can be taken to be a finite sum of $I_{\mu}$'s with $\vert \mu \vert<\vert \lambda \vert$. Let $h \colon I_{\lambda} \to I$ be the composition $I_{\lambda} \to I_{\lambda}/N \to I$. By the universality of $f$, we have $h=gf$ for some $g \colon J \to I$, and so $\ker(f) \subset \ker(h)$. Since $\ker(h)=N$, this shows that $N=L_{\lambda}$, and so $L_{\lambda}$ is simple.

Since $I_{\lambda}$ is indecomposable, it follows that it is the injective envelope of $L_{\lambda}$. Since $L_{\lambda}$ is simple, it is therefore the socle of $I_{\lambda}$.
\end{proof}

\begin{lemma}
Every simple object of $\cA$ is isomorphic to $L_{\lambda}$, for a unique $\lambda$.
\end{lemma}

\begin{proof}
Let $L$ be a simple object of $\cA$. Then $L$ is necessarily finitely generated, and so by (A8) we have an injection $L \to I$, where $I$ is a finite sum of $I_{\lambda}$'s. Since $L$ is simple, it follows that $L$ must inject into one of the factors, and land in the socle. This gives an isomorphism $L \cong L_{\lambda}$.

Suppose now that $L_{\lambda} \cong L_{\mu}$. Then the injective envelopes of $L_{\lambda}$ and $L_{\mu}$ would be isomorphic, i.e., $I_{\lambda} \cong I_{\mu}$. By (A6), this implies that $\lambda=\mu$.
\end{proof}

\begin{lemma} \label{lem:Ilambda-4}
Every object of $\cA$ is locally of finite length.
\end{lemma}

\begin{proof}
By (A1), it suffices to show that every finitely generated object of $\cA$ is finite length. By (A8), it suffices to show that each $I_{\lambda}$ has finite length. We proceed by induction on $\vert \lambda \vert$. Thus suppose $I_{\mu}$ has finite length for $\vert \mu \vert<n$ and let $\lambda$ be given with $\vert \lambda \vert=n$. Using notation as in Lemma~\ref{lem:Ilambda-3}, we have an exact sequence
\begin{displaymath}
0 \to L_{\lambda} \to I_{\lambda} \to \bigoplus_{\vert \mu \vert<n} \Hom(I_{\lambda}, I_{\mu})^* \otimes I_{\mu}.
\end{displaymath}
By Lemma~\ref{lem:Ilambda-3}, the object $L_{\lambda}$ is simple. By induction, each $I_{\mu}$ appearing in the sum on the right has finite length. By (A2), the sum is finite, and by (A4) each $\Hom$ space is finite dimensional. Thus the rightmost term above has finite length. It follows that $I_{\lambda}$ has finite length, as required.
\end{proof}

\begin{lemma} \label{lem:Ilambda-5}
Every indecomposable injective object of $\cA$ is isomorphic to $I_{\lambda}$ for a unique $\lambda$.
\end{lemma}

\begin{proof}
Let $I$ be an indecomposable injective. Since $I$ is the union of its finite length subobjects b Lemma~\ref{lem:Ilambda-4}, it follows that the socle of $I$ is simple, and that $I$ is its injective envelope. Thus $I \cong I_{\lambda}$ for some $\lambda$. This $\lambda$ is unique, as $I_{\lambda} \cong I_{\mu}$ implies $\lambda=\mu$ by (A6).
\end{proof}

\begin{lemma} \label{lem:Ilambda-6}
The simple $L_{\lambda}$ occurs in $I_{\lambda}$ with multiplicity one. The remaining simple constituents of $I_{\lambda}$ have the form $L_{\mu}$ with $\vert \mu \vert<\vert \lambda \vert$.
\end{lemma}

\begin{proof}
We proceed by induction on $\vert \lambda \vert$. Using notation as in Lemma~\ref{lem:Ilambda-3}, we have an exact sequence
\begin{displaymath}
0 \to L_{\lambda} \to I_{\lambda} \to \bigoplus_{\vert \mu \vert<n} \Hom(I_{\lambda}, I_{\mu})^* \otimes I_{\mu}.
\end{displaymath}
The result now follows.
\end{proof}

\begin{lemma} \label{lem:Ilambda-7}
Every finite length object of $\cA$ has finite injective dimension.
\end{lemma}

\begin{proof}
It suffices to prove that each $L_{\lambda}$ has finite injective dimension. We proceed by induction on $\lambda$. Thus suppose that $L_{\mu}$ has finite injective dimension for $\vert \mu \vert<\vert \lambda \vert$. By Lemma~\ref{lem:Ilambda-6}, it thus follows that $I_{\lambda}/L_{\lambda}$ has finite injective dimension, and so $L_{\lambda}$ does as well.
\end{proof}

\subsection{Applications to rational \texorpdfstring{$\GL$}{GL}-fields} \label{ss:rat-struct}

Fix a pure tuple $\usigma$, let $R=\Sym(k^{\oplus \sigma})$, and let $K=\Frac(R)$. For a partition $\lambda$, we let $L_{\lambda}$ be the intersection of the kernels of all maps $K^{\oplus \lambda} \to K^{\oplus \mu}$ with $\vert \mu \vert<\vert \lambda \vert$. The following is our main result on the structure of $K$-modules.

\begin{theorem} \label{thm:rat-struc}
We have the following:
\begin{enumerate}
\item Every finitely generated $K$-module has finite length.
\item The indecomposable injective $K$-modules are exactly the $K^{\oplus \lambda}$, with $\lambda$ a partition.
\item Every finite lengh $K$-module has finite injective dimension.
\item The $K$-module $L_{\lambda}$ is simple, and is the socle of $K^{\oplus \lambda}$. Every simple $K$-module is isomorphic to a unique $L_{\lambda}$.
\item The simple $L_{\lambda}$ occurs in $K^{\oplus \lambda}$ with multiplicity one; the remaining simple constituents have the form $L_{\mu}$ with $\vert \mu \vert<\vert \lambda \vert$.
\end{enumerate}
\end{theorem}

\begin{proof}
We apply Proposition~\ref{prop:Ilambda}. We take $\cA=\Mod_K$, take $\Lambda$ to be the set of partitions, and take $\vert \lambda \vert$ to have its usual meaning (the size of $\lambda$). For $\lambda \in \Lambda$, we let $I_{\lambda}=K^{\oplus \lambda}$. We verify the conditions (A1)--(A8). The first three conditions are clear.

Now, recall from \S \ref{ss:gen} that we have a functor $T \colon \Mod_R \to \Mod_K$ given by $M \mapsto K \otimes_R M$, which has a right adjoint $S \colon \Mod_K \to \Mod_R$ given by $S(N)=N^{\pol}$. Moreover, $S(K^{\oplus \ulambda})=R^{\oplus \ulambda}$ for any tuple $\ulambda$ (Proposition~\ref{prop:sat}). In particular, we have
\begin{displaymath}
\Hom_K(K^{\oplus \lambda}, K^{\oplus \mu})=\Hom_R(R^{\oplus \lambda}, R^{\oplus \mu})=\Hom_{\GL}(k^{\oplus \lambda}, R^{\oplus \mu}).
\end{displaymath}
This is finite dimensional over $k$ since $k^{\oplus \lambda}$ occurs in $R^{\oplus \mu}$ with finite multiplicity; this proves (A4). If $\lambda=\mu$ then we find that the above space is isomorphic to $k$, which proves (A5). Finally, if $\vert \lambda \vert<\vert \mu \vert$, or if $\vert \lambda \vert=\vert \mu \vert$ but $\lambda \ne \mu$, then the above space is~0, which proves (A6).

We now handle (A7). Since $I_{\lambda}$ is finitely generated, it suffices to consider the case where $J$ is a finite direct sum in (A7). Thus, suppose by way of contraction that we have an injection $K^{\oplus \lambda} \to K^{\oplus \umu}$, where $\umu$ is a tuple composed of partitions that are strictly smaller than $\lambda$. Applying the $S$ functor, this gives an injection of $R$-modules $R^{\oplus \lambda} \to R^{\oplus \umu}$. Let $n$ be such that $\dim \bS_{\lambda}(k^n) > \dim \bS_{\umu}(k^n)$. This is possible since $\dim \bS_{\lambda}(k^n)$ is a polynomial in $n$ of degree $\vert \lambda \vert$, while $\dim \bS_{\umu}(k^n)$ is a polynomial of degree $<\vert \lambda \vert$. Evaluating our injection on $k^n$, we obtain an injection
\begin{displaymath}
R\{k^n\} \otimes \bS_{\lambda}(k^n) \to R\{k^n\} \otimes \bS_{\umu}(k^n)
\end{displaymath}
of $R\{k^n\}$-modules. This is impossible, as the two modules above are free of finite rank, and the domain has greater rank. We thus have a contradiction, which proves (A7).

Finally, (A8) is exactly Corollary~\ref{cor:embed}. This completes the verification of (A1)--(A8). Thus Proposition~\ref{prop:Ilambda} applies, which completes the proof.
\end{proof}

\begin{corollary}
All projective $R$-modules are injective.
\end{corollary}

\begin{proof}
Let $S$ and $T$ be as in the above proof. Since $T$ is exact, the its right adjoint $S$ takes injectives to injectives. In particular, we see that $S(K \otimes V)$ is an injective $R$-module for any polynomial representation $V$. As $S(K \otimes V)=R \otimes V$ (Proposition~\ref{prop:sat}), and every projective $R$-module has this form, the result follows.
\end{proof}

\subsection{Applications to other \texorpdfstring{$\GL$}{GL}-fields} \label{ss:genfield}

By leveraging Theorem~\ref{thm:rat-struc}, we are able to deduce the following fundamental result for more general $\GL$-fields:

\begin{theorem} \label{thm:gen-struc}
Let $K$ be a $\GL$-field that is finitely generated over its invariant subfield $k$.
\begin{enumerate}
\item Any finitely generated $K$-module has finite length.
\item If $M$ and $N$ are finitely generated $K$-modules then $\Hom_K(M,N)$ is a finite dimensional $k$-vector space.
\end{enumerate}
\end{theorem}

The first statement is reasonably straightforward:

\begin{proof}[Proof of Theorem~\ref{thm:gen-struc}(a)]
Applying Proposition~\ref{prop:unirat-K}, let $n$ be such that $\Sh_n(K)$ is a rational $\GL$-field over its invariant subfield. Let $V$ be a finitely generated $K$-module. Then $\Sh_n(V)$ is a finitely generated $\Sh_n(K)$-module, and therefore of finite length by Theorem~\ref{thm:rat-struc}(a). It follows that $V$ has finite length. In fact, if $\Sh_n(V)$ has length $\ell$ then $V$ has length $\le \ell$, for if $U_0 \subset \cdots \subset U_{\ell+1}$ is any chain of $K$-submodules of $V$ then $\Sh_n(U_0) \subset \cdots \subset \Sh_n(U_{\ell+1})$ is a chain of $\Sh_n(K)$-submodules of $\Sh_n(V)$, and so $\Sh_n(U_i)=\Sh_n(U_{i+1})$ for some $i$, and so $U_i=U_{i+1}$.
\end{proof}

The second part of the theorem will take the remainder of the section. We require a number of lemmas.

\begin{lemma} \label{lem:struc-1}
Theorem~\ref{thm:gen-struc}(b) holds if $K$ is a rational $\GL$-field over $k$.
\end{lemma}

\begin{proof}
Choose a surjection $K^{\oplus \ulambda} \to M$ for some tuple $\ulambda$, which is possible in general, and an injection $N \to K^{\oplus \umu}$ for some tuple $\umu$, which is possible by the embedding theorem (Corollary~\ref{cor:embed}) since $K$ is rational. We thus obtain an injection
\begin{displaymath}
\Hom_K(M,N) \to \Hom_K(K^{\oplus \ulambda}, K^{\oplus \umu}).
\end{displaymath}
We have seen (in the proof of Theorem~\ref{thm:rat-struc}) that this is finite dimensional over $k$. The result follows.
\end{proof}

\begin{lemma} \label{lem:struc-2}
Let $M$ be a finitely generated $K$-module and let $\phi$ be an endomorphism of $M$. Then $\phi$ satisfies a non-zero polynomial with coefficients in $K$.
\end{lemma}

\begin{proof}
Applying Proposition~\ref{prop:unirat-K}, let $n$ be such that $\Sh_n(K)$ is a rational $\GL$-field over its invariant subfield; in other words, this means $K$ is rational over $K^{G(n)}$ as a $G(n)$-field. Let $E$ be the space of all $K$-linear $G(n)$-equivariant maps $M \to M$; this is identified with $\End_{\Sh_n(K)}(\Sh_n(M))$. By Lemma~\ref{lem:struc-1}, $E$ is a finite dimensional vector space over the field $K^{G(n)}$. Thus the elements $\{\phi^i\}_{i \ge 0}$ of $E$ are linearly dependent, which gives the requisite polynomial.
\end{proof}

For $\phi$ as above, the set of all polynomials that $\phi$ satisfies forms an ideal in the univariate polynomial ring $K[T]$. We define the \emph{minimal polynomial} of $\phi$ to be the unique monic generator of this ideal. In other words, the minimal polynomial of $\phi$ is the unique monic polynomial that $\phi$ satisfies of minimal degree.

\begin{lemma} \label{lem:struc-3}
Let $M$ be a finitely generated $K$-module and let $\phi$ be an endomorphism of $M$. Then the minimal polyomial of $\phi$ has coefficients in the invariant field $k$.
\end{lemma}

\begin{proof}
Suppose that $\sum_{i=0}^d c_i \phi^i=0$ is the equation given by the minimal polynomial. If $g \in \GL$ then we also have $\sum_{i=0}^d (gc_i) \phi^i=0$. By uniquness of the minimal polynomial, we therefore have $gc_i=c_i$. Since this holds for all $g$, it follows that $c_i \in k$, as required.
\end{proof}

The following lemma is a version of Schur's lemma:

\begin{lemma} \label{lem:struc-4}
Suppose that $k$ is algebraically closed and $M$ is a simple $K$-module. Then $\End_K(M)=k$.
\end{lemma}

\begin{proof}
Since $M$ is simple, it follows that $D=\End_K(M)$ is a division ring. We know that $D$ contains $k$ in its center. By the Lemma~\ref{lem:struc-3}, every element of $D$ is algebraic over $k$. (Note that $M$ is necessarily finitely generated since it is simple.) Since $k$ is algebraically closed, it follows that $D=k$.
\end{proof}

\begin{lemma} \label{lem:struc-5}
Suppose that $k$ is algebraically closed. Then Theorem~\ref{thm:gen-struc}(b) holds.
\end{lemma}

\begin{proof}
It follows from the previous lemma that $\Hom_K(M,N)$ is finite dimensional over $k$ if $M$ and $N$ are simple. As $M$ and $N$ have finite length by Theorem~\ref{thm:gen-struc}(a), the general case follows from d\'evissage.
\end{proof}

We now deduce the general case from the case with $k$ algebraically closed using a base change argument. For this, we require two more lemmas.

\begin{lemma} \label{lem:struc-6}
Any element of $K$ that is algebraic over $k$ belongs to $k$, i.e., $k$ is algebraically closed within $K$.
\end{lemma}

\begin{proof}
Let $a \in K$ be algebraic over $k$, and let $f(T) \in k[T]$ be its minimal polynomial. Since $\GL$ acts on $K$ by field homomorphisms, it permutes the roots of $f$ in $K$. This action corresponds to a homomorphism $\phi \colon \GL \to \fS_n$ where $n$ is the number of roots of $f$ in $K$. Since any group homomorphism $\bQ \to \fS_n$ is trivial, it follows that $\phi$ is trivial on each group of elementary matrices in $\GL$. Since these groups generate $\SL$, it follows that $\phi(\SL)=1$. We thus see that $a$ is fixed by $\SL$. However, $a$ is also fixed by $G(m)$ for $m \gg 0$. It follows that $a$ is fixed by $\GL=\SL \cdot G(m)$, i.e., $a \in k$.
\end{proof}

Suppose $k'$ is an algebraic extension of $k$. Then the above lemma implies that $K'=k' \otimes_k K$ is a field. Letting $\GL$ act on $K'$ by acting trivially on $k'$, one easily sees that $K'$ is a $\GL$-field, its invariant field is $k'$, and it is finitely $\GL$-generated over $k'$.

\begin{lemma} \label{lem:struc-7}
Let $k'$ be an algebraic extension of $k$ and put $K'=k' \otimes_k K$. Let $M$ and $N$ be $K$-modules, with $M$ finitely generated. Then the natural map
\begin{displaymath}
k' \otimes_k \Hom_K(M,N) \to \Hom_{K'}(k' \otimes_k M, k' \otimes_k N)
\end{displaymath}
is an isomorphism.
\end{lemma}

\begin{proof}
By adjunction, we have
\begin{displaymath}
\Hom_{K'}(k' \otimes_k M, k' \otimes_k N) = \Hom_K(M, k' \otimes_k N).
\end{displaymath}
Now, for any $k$-vector space $E$, we have a natural map
\begin{displaymath}
E \otimes_k \Hom_K(M,N) \to \Hom_K(M, E \otimes_k N).
\end{displaymath}
Picking a $k$-basis $\{e_i\}_{i \in I}$ for $E$, we find that the above map is isomorphic to the map
\begin{displaymath}
\Hom_K(M,N)^{\oplus I} \to \Hom_K(M, N^{\oplus I}).
\end{displaymath}
This map is an isomorphism since $M$ is finitely generated. Applying this with $E=k'$ gives the result.
\end{proof}

\begin{proof}[Proof of Theorem~\ref{thm:gen-struc}(b)]
Let $k'$ be an algebraic closure of $k$, and let $K'=k' \otimes_k K$. Let $M$ and $N$ be finitely generated $K$-modules. By Lemma~\ref{lem:struc-7}, the map
\begin{displaymath}
k' \otimes_k \Hom_K(M,N) \to \Hom_{K'}(k' \otimes_k M, k' \otimes_k N)
\end{displaymath}
is an isomorphism. As $k' \otimes_k M$ and $k' \otimes_k N$ are finitely generated $K'$-modules, it follows from Lemma~\ref{lem:struc-5} that the right side above is a finite dimensional $k'$-vector space. It thus follows that $\Hom_K(M,N)$ is a finite dimensional $k$-vector space, which completes the proof.
\end{proof}

\section{Brauer categories, Weyl's construction, universal properties} \label{s:brauer}

The purpose of this section is to describe $\Mod_K$, when $K$ is a rational $\GL$-field, in terms of a combinatorial category, the upwards $\usigma$-Brauer category $\fU(\usigma)$. We begin in \S \ref{ss:repcat} by reviewing generalities on representations of categories. We introduce $\fU(\usigma)$ in \S \ref{ss:brauer}. The main equivalences are established in \S \ref{ss:equiv}. Finally, in \S \ref{ss:weyl} and \S \ref{ss:universal}, we give applications of these equivalences: we establish a version of Weyl's traceless tensor construction for $\Mod_K$, and give a universal property for $\Mod_K$.

\subsection{Representations of categories} \label{ss:repcat}

We now review a bit of material on representations of categories. See \cite[\S 3]{brauercat1} for more detail.

Let $\fG$ be an essentially small $k$-linear category. A \defn{representation of $\fG$}, or a \defn{$\fG$-module}, is a functor $\fG \to \Vec_k$, and a \defn{map} of $\fG$-modules is a natural transformation. We let $\Mod_{\fG}$ be the category of $\fG$-modules. For $\fG$-modules $M$ and $N$, we write $\Hom_{\fG}(M,N)$ for the set of maps of $\fG$-modules $M \to N$.

Let $x$ be an object of $\fG$. We define the \defn{principal projective} $\fG$-module at $x$, denoted $\bP_x$, by $\bP_x(y)=\Hom_{\fG}(x,y)$. If $M$ is an arbitary $\fG$-module then we have an identification
\begin{displaymath}
\Hom_{\fG}(\bP_x, M) = M(x)
\end{displaymath}
by Yoneda's lemma, which shows that $\bP_x$ is projective. The above identity also shows that $M$ can be realized as a quotient of a direct sum of principal projectives.

We similarly define the \defn{principal injective} $\fG$-module at $x$, denoted $\bI_x$, by $\bI_x(y)=\Hom_{\fG}(y,x)^*$. If $M$ is an arbitary $\fG$-module then we have an identification
\begin{displaymath}
\Hom_{\fG}(M, \bI_x) = M(x)^*
\end{displaymath}
(see \cite[Proposition~3.2]{brauercat1}), which shows that $\bI_x$ is injective.

\begin{proposition} \label{prop:inj-proj-eq}
Suppose that the $\Hom$ sets in $\fG$ are finite dimensional. Then the following categories are equivalent:
\begin{enumerate}
\item The category $\fG^{\op}$.
\item The full subcategory of $\Mod_{\fG}$ spanned by the principal projectives.
\item The full subcategory of $\Mod_{\fG}$ spanned by the principal injectives.
\end{enumerate}
\end{proposition}

\begin{proof}
Let $\cC$ be the category in (b). We have a functor $\fG^{\op} \to \cC$ given by $x \mapsto \bP_x$. It is obviously essentially surjective, and is fully faithful by Yoneda's lemma. Similarly, let $\cC'$ be the category in (c). Then we have a functor $\fG^{\op} \to \cC'$ given by $x \mapsto \bI_x$. We have
\begin{displaymath}
\Hom_{\fG}(\bI_x, \bI_y)=\bI_x(y)^* = \Hom_{\fG}(y,x)^{**} = \Hom_{\fG}(y,x) = \Hom_{\fG^{\op}}(x,y).
\end{displaymath}
One easily sees that this identification is induced by the functor under consideration, which shows that it is fully faithful.
\end{proof}

Suppose now that the isomorphism classes of $\fG$ are in bijection with the set $\bN$ of natural numbers; for $n \in \bN$, we let $[n]$ be a representative of the $n$th isomorphism class. We say that $\fG$ is \emph{upwards} if $\Hom_{\fG}([n], [m])\ne 0$ implies $n \le m$.

\begin{proposition} \label{prop:prin-inj}
Suppose $\fG$ is upwards and all $\Hom$ sets are finite dimensional. Then the principal injectives are of finite length, and every finite length $\fG$-module embeds into a finite sum of principal injectives.
\end{proposition}

\begin{proof}
Let $M$ be a $\fG$-module, and write $M_n$ in place of $M([n])$. Define the \emph{support} of $M$ to be the set of natural numbers $n$ for which $M_n \ne 0$. Define the $n$th \emph{truncation} of $M$, denoted $\tau_{\ge n}(M)$, to be the $\fG$-module given by
\begin{displaymath}
\tau_{\ge n}(M)_m = \begin{cases}
M_m & \text{if $m \ge n$} \\
0 & \text{if $m<n$} \end{cases};
\end{displaymath}
one easily sees that this is a $\fG$-submodule of $M$ since $\fG$ is upwards. From the above structure, one easily verifies the following two statements:
\begin{enumerate}
\item A $\fG$-module $M$ is simple if and only if it is supported in a single degree $n$ and $M_n$ is a simple module over the ring $\End_{\fG}([n])$.
\item A $\fG$-module $M$ has finite length if and only if it has finite support and $M_n$ is finite dimensional for all $n$.
\end{enumerate}
It follows from (b) that the principal injective $\bI_n$ is of finite length. It follows from (a) that if $M$ is a simple supported in degree $n$ then $M$ embeds into $\bI_n$. One now easily sees that any finite length objects embeds into a sum of $\bI_n$'s.
\end{proof}

\subsection{A variant of the Brauer category} \label{ss:brauer}

The upwards and downwards Brauer categories were introduced in \cite[\S 4.2.5]{infrank} as a means to describe the category of algebraic representations of the infinite orthogonal group. We now introduce a generalization that will similarly allow us to describe the category of $K$-modules.

For a partition $\lambda$ of $n$, recall that $S^{\lambda}$ is the irreducible representation of $\fS_n$ associated to $\lambda$ (the Specht module). For a finite set $A$ of cardinality $n$, we let $S^{\lambda}_A$ be the associated representation of $\Aut(A) \cong \fS_n$. One can define this in a canonical manner by mimicking the construction of $S^{\lambda}$, but using elements of $A$ in place of the integers $1, \ldots, n$.

Fix a pure tuple $\usigma=[\sigma_1, \ldots, \sigma_r]$. A \defn{$\usigma$-block} on a set $S$ is a triple $(A, p, x)$ where
\begin{itemize}
\item $p$ is an element of $[r]$,
\item $A$ is a subset of $S$ of cardinality $\sigma_p$ (called the \defn{support} of the block),
\item $x$ is an element of the Specht module $S^{\sigma_p}_A$.
\end{itemize}
Let $S$ and $T$ be a finite sets. A \defn{downwards $\usigma$-diagram} from $S$ to $T$ is a pair $(i, \Gamma)$ where $\Gamma$ is a collection of $\usigma$-blocks on $S$ with disjoint supports and $i \colon S \setminus \vert \Gamma \vert \to T$ is a bijection, where $\vert \Gamma \vert$ is the union of the supports of the blocks in $\Gamma$. The \defn{space of downwards $\usigma$-diagrams} is the vector space spanned by elements $[i,\Gamma]$, with $(i,\Gamma)$ an downwards $\usigma$-diagram, with the following relation imposed:
\begin{itemize}
\item Suppose that $\Gamma$ contains a block $(A, p, x)$, and let $x=\alpha y+\beta z$ be a linear combination in the Specht module. Let $\Gamma'$ be the diagram obtained by replacing this block with $(A, p, y)$, and let $\Gamma'$ be defined similarly but using $z$. Then $[i,\Gamma]=\alpha [i,\Gamma']+\beta [i,\Gamma'']$.
\end{itemize}
We now come to the main definition:

\begin{definition}
The \defn{downwards $\usigma$-Brauer category}, denoted $\fD(\usigma)$, is the $k$-linear category described as follows.
\begin{itemize}
\item The objects of $\fD(\usigma)$ are finite sets.
\item Given finite sets $S$ and $T$, the space of morphisms $\Hom_{\fD(\usigma)}(S, T)$ is the space of downwards $\usigma$-diagrams from $S$ to $T$.
\item Composition is defined as follows. Let $(i, \Gamma)$ be a diagram from $S$ to $T$, and let $(i', \Gamma')$ be a diagram from $T$ to $U$. Let $j=i' \circ i$ and let $\Delta=\Gamma \sqcup i^{-1}(\Gamma')$, where $i^{-1}(\Gamma')$ denotes the result of transporting $\Gamma'$ along the bijection $i^{-1} \colon \vert \Gamma' \vert \to i^{-1}(\vert \Gamma' \vert)$. Then $[i', \Gamma'] \circ [i, \Gamma]=[j, \Delta]$. \qedhere
\end{itemize}
\end{definition}

\begin{example}
If $\usigma=[(2)]$ then $\fD(\usigma)$ is the downwards Brauer category from \cite[\S 4.2.5]{infrank}. Similarly, if $\sigma=[(1,1)]$ then $\fD(\usigma)$ is the signed downwards Brauer category, also discussed in \cite[\S 4.2.11]{infrank}.
\end{example}

\begin{example}
Suppose that $\usigma=[\sigma_1, \ldots, \sigma_d]$ where $\sigma_i=(1)$ for all $i$. Then a downwards $\usigma$-diagram from $S$ to $T$ is an injection $f \colon T \to S$ together with a $d$-coloring on $S \setminus \im(f)$. We thus see that $\fD(\usigma)$ is the opposite of the category $\FI_d$ introducing in \cite[\S 7]{catgb}. (The category $\FI_1$ is just the category $\FI$ of finite sets and injections, as in \cite{fimodule}.)
\end{example}

The category $\fD(\usigma)$ carries a natural symmetric monoidal structure $\amalg$ given by disjoint union. Precisely, for two objects $S$ and $T$, the object $S \amalg T$ is simply the disjoint union of the sets $S$ and $T$. Given two morphisms $[i, \Gamma] \colon S \to T$ and $[i', \Gamma'] \colon S' \to T'$, the morphism $[i, \Gamma] \amalg [i', \Gamma']$ is defined to be $[i \amalg i', \Gamma \amalg \Gamma']$. Note that $\amalg$ is a $k$-linear functor in each of its arguments.

The category $\fD(\usigma)$ admits a universal property, which we now describe. Let $\cC$ be an $k$-linear symmetric monoidal category. Let $T_{\usigma}(\cC)$ be the category whose objects are pairs $(X, \omega)$, where $X$ is an object of $\cC$ and $\omega \colon \bS_{\usigma}(X) \to \bone$ is a morphism in the Karoubian--additive envelope of $\cC$, where $\bone$ is the unit object of $\cC$; of course, if $\cC$ is additive and Karoubian (e.g., abelian) then one does not need to take the envelope here. Morphisms in $T_{\usigma}(\cC)$ are defined in the obvious manner.

\begin{proposition} \label{prop:Duniv}
Notation as above, we have a natural equivalence of categories
\begin{displaymath}
\Phi \colon \Fun^{\otimes}_k(\fD(\usigma), \cC) \to T_{\usigma}(\cC).
\end{displaymath}
Here $\Fun^{\otimes}_k(-, -)$ denotes the category of symmetric monoidal $k$-linear functors.
\end{proposition}

\begin{proof}
We first define the functor $\Phi$. Thus suppose given a symmetric monoidal $k$-linear functor $\theta \colon \fD(\usigma) \to \cC$. Let $X=\theta([1])$. We define a map $\omega \colon \bS_{\usigma}(X) \to \bone$. It suffices to define maps $\omega_p \colon \bS_{\sigma_p}(X) \to \bone$ for each $p \in [r]$. Thus fix such $p$. Put $n_p=\vert \sigma_p \vert$. Since $\theta$ is a symmetric monoidal functor, it induces an $\fS_{n_p}$-equivariant map
\begin{displaymath}
\Hom_{\fD(\usigma)}([1]^{\otimes n_p}, [0]) \to \Hom_{\cC}(X^{\otimes n_p}, \bone).
\end{displaymath}
Now, $[1]^{\otimes n_p}=[n_p]$. Inside of $\Hom_{\fD(\usigma)}([n_p], [0])$ one has the subspace spanned by diagrams that consist of a single block of type $p$. This subspace is isomorphic to $S^{\sigma_p}$ as a representation of $\fS_{n_p}$. We thus obtain a canonical $\fS_{n_p}$-equivariant map $S^{\sigma_p} \to \Hom_{\cC}(X^{\otimes n_p}, \bone)$, which yields a map $\bS_{\sigma_p}(X) \to \bone$, as required. We have thus defined $\omega$. We define $\Phi$ on objects by $\Phi(\theta)=(X,\omega)$. The definition on morphisms is clear.

To show that $\Phi$ is an equivalence, we construct a quasi-inverse functor
\begin{displaymath}
\Psi \colon T_{\usigma}(\cC) \to \Fun^{\otimes}_k(\fD(\usigma), \cC).
\end{displaymath}
Thus let $(X, \omega)$ in $T_{\usigma}(\cC)$ be given. We define a symmetric monoidal $k$-linear functor $\theta \colon \fD(\usigma) \to \cC$. On objects, we define $\theta$ by $\theta(S)=X^{\otimes S}$. Now, consider a $\usigma$-block $(A, p, x)$. We have
\begin{displaymath}
\Hom_{\cC}(X^{\otimes A}, \bone) = \bigoplus_{\mu \vdash n_p} S^{\mu}_A \otimes \Hom_{\cC}(\bS_{\mu}(X), \bone).
\end{displaymath}
The $\mu=\sigma_p$ summand on the right side contains the element $x \otimes \omega_p$. We say that the corresponding morphism $X^{\otimes A} \to \bone$ is \defn{associated} to this block. Note that this construction is linear in the element $x$. Now, consider a morphism $f \colon S \to T$ in $\fD(\usigma)$ represented by a diagram. Suppose this diagram corresponds to a pair $(i, \Gamma)$, where $\Gamma$ is a collection of disjoint blocks on $S$ and $i \colon S \setminus \vert \Gamma \vert \to T$ is a bijection. We define a morphism $\theta(f) \colon X^{\otimes S} \to X^{\otimes T}$ as follows. Write $X^{\otimes S}=X^{\otimes \vert \Gamma \vert} \otimes X^{\otimes S \setminus \vert \Gamma \vert}$. We have a map $X^{\otimes \vert \Gamma \vert} \to \bone$ by tensoring together the maps associated to individual blocks. We also have a map $X^{\otimes S \setminus \vert \Gamma \vert} \to X^{\otimes T}$ from the bijection $i$. The map $\theta(f)$ is the tensor product of these two maps. The construction $\theta$ extends to a $k$-linear map
\begin{displaymath}
\theta \colon \Hom_{\fD(\usigma)}(S, T) \to \Hom_{\cC}(X^{\otimes S}, X^{\otimes T}).
\end{displaymath}
One easily verifies that $\theta$ is compatible with composition and is naturally a symmetric monoidal functor. We define $\Psi$ on objects by $\Psi(X,\omega)=\theta$. The definition of $\Psi$ on morphisms is clear.

One easily verifies that $\Phi$ and $\Psi$ are naturally quasi-inverse. This completes the proof.
\end{proof}

There is also an \defn{upwards $\usigma$-Brauer category} $\fU(\usigma)$, defined in the same manner, but where now blocks are only allowed on the target of a morphism. In other words, $\fU(\usigma)$ is simply the opposite category of $\fD(\usigma)$. The category $\fU(\usigma)$ admits a natural symmetric monoidal structure, and has a similar universal property to $\fD(\usigma)$.

\subsection{Equivalences} \label{ss:equiv}

We now establish a number of equivalences between categories associated to $R$, $K$, and the $\usigma$-Brauer categories..

\begin{proposition} \label{prop:equivs}
The following symmetric monoidal $k$-linear categories are equivalent:
\begin{enumerate}
\item The downwards $\usigma$-Brauer category $\fD(\usigma)$.
\item The full subcategory of $\Mod_R$ spanned by the objects $R \otimes \bV^{\otimes n}$ for $n \ge 0$.
\item The full subcategory of $\Mod_K$ spanned by the objects $K \otimes \bV^{\otimes n}$ for $n \ge 0$.
\end{enumerate}
As $k$-linear categories (ignoring the monoidal structure), these categories are also equivalent to
\begin{enumerate}[resume]
\item The full subcategory of $\Mod_{\fU(\usigma)}$ spanned by the principal projective objects.
\item The full subcategory of $\Mod_{\fU(\usigma)}$ spanned by the principal injective objects.
\end{enumerate}
\end{proposition}

\begin{proof}
We break the proof into three steps.

\textit{Step 1: equivalence of (a) and (b).} Let $\cC$ be the category in (b) and let $X=R \otimes \bV$. Then $\bS_{\usigma}(X)=R \otimes \bS_{\usigma}(\bV)$, where on the left side $\bS_{\usigma}$ is formed with respect to $\otimes_R$. Since $R$ contains $\bS_{\usigma}(\bV)$ as a subrepresentation, there is a natural map of $R$-modules $R \otimes \bS_{\usigma}(\bV) \to R$. We thus have a natural map $\omega \colon \bS_{\usigma}(X) \to R$. Since $R$ is the unit object for $\otimes_R$, the universal property of $\fD(\usigma)$ (Proposition~\ref{prop:Duniv}) furnishes a symmetric monoidal $k$-linear functor
\begin{displaymath}
\theta \colon \fD(\usigma) \to \cC.
\end{displaymath}
This functor has the property that $\theta([n])=X^{\otimes n} = R \otimes \bV^{\otimes n}$. It is clear that $\theta$ is essentially surjective. To complete this step, it suffices to show that $\theta$ is fully faithful.

Before doing this, we introduce some notation. Identify the weight lattice of the diagonal torus in $\GL$ with $\bZ^{\oplus \infty}$. For a finite subset $A \subset [\infty]$, let $1^A$ denote the weight that is~1 in the $A$ coordinates and~0 elsewhere; also, write $1^n$ in place of $1^{[n]}$. Given a weight $\lambda$ and a polynomial representation $V$, let $V_{\lambda}$ be the $\lambda$ weight space of $V$.

Now, we have
\begin{displaymath}
\Hom_R(X^{\otimes n}, X^{\otimes m})
=\Hom_R(R \otimes \bV^{\otimes n}, R \otimes \bV^{\otimes m})
=\Hom_{\GL}(\bV^{\otimes n}, R \otimes \bV^{\otimes m}).
\end{displaymath}
By Schur--Weyl duality, $\Hom_{\GL}(\bV^{\otimes n}, W)=W_{1^n}$ for any polynomial representation $W$; explicitly, a map $\phi \colon \bV^{\otimes n} \to W$ corresponds to $\phi(e_1 \otimes \cdots \otimes e_n) \in W_{1^n}$. We must therefore understand the $1^n$ weight space of $R \otimes \bV^{\otimes m}$.

Let $p \in [r]$. The $1^n$-weight space of $\bS_{\sigma_p}(\bV)$ is canonically isomorphic to the Specht module $S^{\sigma_p}$ if $n=\# \sigma_p$, and vanishes for other values of $n$. More generally, let $A$ be a subset of $[\infty]$ of size $\# \sigma_p$. Then we have a canonical isomorphism $(\bS_{\sigma_p}(\bV))_{1^A}=S^{\sigma_p}_A$. Fix a basis $\cS^{\sigma_p}_A$ for $S^{\sigma_p}_A$. For $x \in \cS^{\sigma_p}_A$, let $t_{A,p,x} \in (\bS_{\sigma_p}(\bV))_{1^A}$ be the image of $x$ under this isomorphism. We refer to $A$ as the \defn{support} of the element $t_{A,p,x}$. Let $\cT \subset R$ be the set of all elements of the form $t_{A,p,x}$ for all choices of $A$, $p$, and $x$, and let $\cM_A$ be the set of all monomials $t_1 \cdots t_s$ where the $t_i$'s belong to $\cT$ and their supports form a partition of $A$. We thus see that $\cM_A$ is a basis for $R_{1^A}$.

From the above discussion, we see that the $1^n$-weight space of $R \otimes \bV^{\otimes m}$ has for a basis all elements of the form
\begin{displaymath}
T \otimes e_{s_1} \otimes \cdots \otimes_{s_m}
\end{displaymath}
where $s_1, \ldots, s_m$ are distinct elements of $[m]$ and $T \in \cM_A$ with $A=[n] \setminus \{s_1, \ldots, s_m\}$. We associate to the above element the $\usigma$-diagram given by the pair $(i, \Gamma)$, where $\Gamma$ is the collection of blocks corresponding to $T$ (each $t_{A,p,x}$ corresponds to a block $(A,p,x)$), and $i \colon [n] \setminus A \to [m]$ is the bijection taking $s_j$ to $j$. We have thus constructed a natural linear isomorphism
\begin{displaymath}
(R \otimes \bV^{\otimes m})_{1^n} = \Hom_{\fD(\usigma)}([n], [m]).
\end{displaymath}
As we have already seen, the left side above is identified with $\Hom_R(X^{\otimes n}, X^{\otimes m})$. One easily sees that the resulting isomorphism
\begin{displaymath}
\Hom_{\fD(\usigma)}([n], [m]) = \Hom_R(X^{\otimes n}, X^{\otimes m})
\end{displaymath}
is induced by $\theta$. This shows that $\theta$ is fully faithful. This completes the first step of the proof.

\textit{Step 2: equivalence of (b) and (c).} Let $\cC'$ be the category in (c). The functor $\Mod_R \to \Mod_K$ given by $M \mapsto K \otimes_R M$ induces a functor $\cC \to \cC'$. This functor is clearly symmetric monoidal, faithful, and essentially surjective. It is full by Proposition~\ref{prop:sat}. Thus it is an equivalence.

\textit{Step 3: the remainder.} To complete the proof, it suffices to show that the categories in (a), (d), and (e) are equivalent, as $k$-linear categories. This follows from Proposition~\ref{prop:inj-proj-eq}.
\end{proof}

For an abelian category $\cA$, we let $\cA^{\lf}$ be the full subcategory spanned by objects that are locally of finite length (i.e., the union of their finite length subobjects).

\begin{proposition}
We have the following equivalences of $k$-linear abelian categories:
\begin{enumerate}
\item $\Mod_R \cong \Mod_{\fU(\usigma)}$
\item $\Mod_K \cong \Mod_{\fU(\usigma)}^{\lf}$
\item $\Mod_K \cong \Mod_R^{\lf}$.
\end{enumerate}
\end{proposition}

\begin{proof}
Let $\cC$ and $\cC'$ be Grothendieck abelian categories, and let $\cP$ and $\cP'$ be full subcategories of $\cC$ and $\cC'$ consisting of projective objects. Suppose that $\cP$ and $\cP'$ are enough projectives (i.e., they form generating families). Then any equivalence $\cP \to \cP'$ extends uniquely to an equivalence $\cC \to \cC'$. A similar statement holds for categories of injective objects.

Statement~(a) now follows from the equivalence between the categories (b) and (d) in Proposition~\ref{prop:equivs}; it is clear that the categories in (b) and (d) are enough projectives in $\Mod_R$ and $\Mod_{\fU(\usigma)}$. Statement~(b) follows from the equivalence between the categories (c) and (e) in Proposition~\ref{prop:equivs}; the fact that category (c) gives enough injectives in $\Mod_K$ follows from Theorem~\ref{thm:rat-struc}, while the fact that category (d) gives enough injectives in $\Mod_{\fU(\usigma})^{\lf}$ is Proposition~\ref{prop:prin-inj} (note that $\fU(\usigma)$ is an upwards category, as defined before Proposition~\ref{prop:prin-inj}). Statement~(c) follows from statements (a) and (b).
\end{proof}

\begin{remark} \label{rmk:gen-tor}
The equivalence $\Mod_K \cong \Mod_R^{\lf}$ has previously been established for a few values of $\usigma$: for $[(1)]$ in \cite{symc1}, for $[(1),\ldots,(1)]$ in \cite{symu1}, for $[(2)]$ and $[(1,1)]$ in \cite{sym2noeth}, and for $[(1,1),(1)]$ in \cite{symc1sp}. Related results also appear in \cite{periplectic} and \cite{isomeric}.
\end{remark}

For a partition $\lambda$, recall that $L_{\lambda}$ denote the simple object of $\Mod_K$ indexed by $\lambda$. Using the above proposition, we can compute the $\Ext$ groups between these objects:

\begin{corollary}
We have
\begin{displaymath}
\Ext^i_K(L_{\lambda}, L_{\mu})= \Hom_{\GL}(\lw^i(k^{\oplus \usigma}) \otimes k^{\oplus \lambda}, k^{\oplus \mu}).
\end{displaymath}
\end{corollary}

\begin{proof}
Let $\cA=\Mod_{\fU(\usigma)}$ and let $\Phi \colon \Mod_K \to \cA^{\lf}$ be the equivalence constructed above. Tracing through the definition, we see that $\Phi$ takes $K \otimes \bV^{\otimes n}$ to the principal injective $\bI_n$. We thus see that $\Phi(K^{\oplus \lambda})$ is the $S^{\lambda}$-isotypic piece of $\bI_n$, with respect to its natural $\fS_n$-action. Taking socles, we see that $L'_{\lambda}=\Phi(L_{\lambda})$ is the simple $\fU(\sigma)$-module with $(L'_{\lambda})_n=S^{\lambda}$. We thus have
\begin{displaymath}
\Ext^i_K(L_{\lambda}, L_{\mu})=\Ext^i_{\cA^{\lf}}(L'_{\lambda}, L'_{\mu}).
\end{displaymath}
The $\Ext$ on the right side can be computed by taking an injective resolution of $L'_{\mu}$ in $\cA^{\lf}$. As we have seen (Proposition~\ref{prop:prin-inj}), this can be accomplished using principal injectives. As these objects are injective in the larger category $\cA$, we find
\begin{displaymath}
\Ext^i_{\cA^{\lf}}(L'_{\lambda}, L'_{\mu})=\Ext^i_{\cA}(L'_{\lambda}, L'_{\mu}).
\end{displaymath}
We now appeal to the equivalence $\cA=\Mod_R$. One easily sees that $L'_{\lambda}$ corresponds to the simple $R$-module $k^{\oplus \lambda}$ (with positive degree elements of $R$ acting by~0). We thus have
\begin{displaymath}
\Ext^i_{\cA}(L'_{\lambda}, L'_{\mu})=\Ext^i_R(k^{\oplus \lambda}, k^{\oplus \mu}).
\end{displaymath}
The right group can be computed using the projective resolution of $k^{\oplus \lambda}$ provided by the Koszul complex. This yields the stated result.
\end{proof}

\subsection{Weyl's construction} \label{ss:weyl}

We recall Weyl's classical traceless tensor construction. Equip $\bC^r$ with a non-degenerate symmetric bilinear form. Let $T^n=(\bC^r)^{\otimes n}$. Given $1 \le i<j \le n$, let $\phi_{i,j} \colon T^n \to T^{n-2}$ be the map obtained by applying the form to the $i$ and $j$ tensor factors. Let $T^{[n]}$ be the intersection of the kernels of $\phi_{i,j}$, over all choices of $i$ and $j$; this is the space of \defn{traceless tensors}. The space $T^{[n]}$ is a $(\fS_n \times \bO_r)$-subrepresentation of $T^n$. Weyl proved that the $S^{\lambda}$ isotypic piece of $T^{[n]}$ is either~0 or the irreducible of $\bO_r$ with highest weight $\lambda$.

We now establish an analog of this construction for $\Mod_K$. Recall that $\usigma=[\sigma_1, \ldots, \sigma_r]$. For each $1 \le i \le r$, let $\phi_i \colon K^{\oplus \sigma_i} \to K$ be the natural map (coming from the inclusion $k^{\oplus \sigma_i} \subset K$). Given an element $x$ of the Specht module $S^{\sigma_i}$, let $\phi_{i,x}$ be the composition
\begin{displaymath}
K \otimes \bV^{\otimes \vert \sigma_i \vert} \to K \otimes \bS_{\sigma_i}(\bV) \to K
\end{displaymath}
where the first map comes from the projection $\bV^{\otimes \vert \sigma_i \vert} \to \bS_{\sigma_i}(\bV)$ provided by $x$, and the second map is $\phi_i$. Let $T^n=K \otimes \bV^{\otimes n}$. Given $1 \le i \le n$, $x \in S^{\sigma_i}$, and a subset $S$ of $[n]$ of size $\vert \sigma_i \vert$, we let $\phi_{i,x,S} \colon T^n \to T^{n-\vert \sigma_i \vert}$ be the map obtained by applying $\phi_{i,x}$ to the tensor factors indexed by $S$. Let $T^{[n]}$ be the intersection of the kernels of the $\phi_{i,x,S}$ over all choices of $i$, $x$, and $S$. This is a $K$-module equipped with an action of $\fS_n$. The following is our analog of Weyl's construction:

\begin{proposition}
Let $\lambda$ be a partition of $n$. Then the $S^{\lambda}$ isotypic piece of $T^{[n]}$ is the simple $K$-module $L_{\lambda}$.
\end{proposition}

\begin{proof}
Under the equivalence $\Mod_K=\Mod_{\fU(\ulambda)}^{\lf}$, the $K$-module $T^n$ corresponds to the $n$th principal injective $\fU(\ulambda)$-module. Thinking in terms of $\ulambda$-diagrams, we see that any map $T^n \to T^m$, with $m<n$, is a linear combination of maps of the form $f \circ \phi_{i,x,S}$, where $f$ is some map. It follows that $T^{[n]}$ is the intersection of the kernels of all maps $T^n \to T^m$ with $m<n$. From this, we see that the $S^{\lambda}$-isotypic piece of $T^{[n]}$ is the intersection of the kernels of alls maps $K^{\oplus \lambda} \to K^{\oplus \mu}$ with $\vert \mu \vert < \vert \lambda \vert$. This is the simple object $L_{\lambda}$ (see \S \ref{ss:rat-struct}).
\end{proof}

\subsection{Universal properties} \label{ss:universal}

We can now give the universal property for the category $\Mod_K$. This is analogous to the universal property for $\Rep(\bO)$ given in \cite[\S 4.4]{infrank}. For symmetric monoidal $k$-linear abelian categories $\cC$ and $\cD$, we let $\LEx^{\otimes}_k(\cC, \cD)$ be the category of left-exact symmetric monoidal $k$-linear functors $\cC \to \cD$. Also, recall the category $T_{\usigma}(\cC)$ defined before Proposition~\ref{prop:Duniv}.

\begin{theorem}
Let $(\cC, \otimes)$ be a symmetric monoidal $k$-linear abelian category with $\otimes$ exact. Then we have a natural equivalence of categories
\begin{displaymath}
\LEx^{\otimes}_k(\Mod_K^{\rf}, \cC) \cong T_{\usigma}(\cC).
\end{displaymath}
In other words, to give a $k$-linear left-exact symmetric monoidal functor $\Mod_K^{\rf} \to \cC$ is the same as to give an object of $\cC$ equipped with a $\usigma$-form.
\end{theorem}

\begin{proof}
Let $\cI$ be the full subcategory of $\Mod_K^{\rf}$ spanned by the objects $K \otimes \bV^{\otimes n}$ for $n \ge 0$. This category is stable under tensor products. As a $k$-linear symmetric monoidal category, it is equivalent to $\fD(\usigma)$ by Proposition~\ref{prop:equivs}. Thus by the universal property for $\fD(\usigma)$ (Proposition~\ref{prop:Duniv}), we have a natural equivalence
\begin{displaymath}
\Fun^{\otimes}_k(\cI, \cC) \cong T_{\usigma}(\cC).
\end{displaymath}
Now, every object of $\cI$ is injective in $\Mod_K$ (Theorem~\ref{thm:rat-struc}(b)), and every object of $\Mod_K^{\rf}$ embeds into a finite direct sum of objects in $\cI$ (Theorem~\ref{thm:embed}). It follows that any functor $\cI \to \cC$ extends uniquely to a left-exact functor $\Mod_K^{\rf} \to \cC$. Since $\cI$ is stable under tensor products, and all tensor products are exact, it follows that this extended functor is symmetric monoidal if the original functor is. This completes the proof.
\end{proof}

\begin{remark} \label{rmk:special}
Let $V$ be a finite dimensional $k$-vector space equipped with a form $\omega \colon \bS_{\usigma}(V) \to k$. From the universal property, get a left-exact cocontinuous symmetric monoidal functor
\begin{displaymath}
\Gamma \colon \Mod_K \to \Vec_k
\end{displaymath}
that we call the \defn{specialization functor} with respect to $V$ and $\omega$. Since $\Gamma$ is left-exact, one can consider its right derived functors $\rR^i \Gamma$, which we call the \defn{derived specialization functors}. Is it possible to compute the values of these functors on simple objects for a generic form $\omega$? When $\usigma=[(2)]$ the category $\Mod_K$ is equivalent to the category of algebraic representations of the infinite orthogonal group (see \cite[Theorem~3.1]{sym2noeth}), as studied in \cite{infrank}, and the derived specialization of simple objects was computed in \cite{lwood}.
\end{remark}

\section{Classification of fiber functors} \label{s:fiber}

In this section, we introduce the notion of a fiber functor on $\Mod_K$, and give a complete classification of them.

\subsection{Definitions}

Fix, for the duration of \S \ref{s:fiber}, a $\GL$-field $K$ that is finitely generated over its invariant subfield $k$, and a $\GL$-algebra $R$ finitely generated over $R_0=k$ with $\Frac(R)=K$. Furthermore, let $X=\Spec(R)$ be the $\GL$-variety associated to $R$. The following is the main object of study in this section:

\begin{definition}
A \defn{fiber functor} on $\Mod_K$ is a symmetric monoidal functor $\Phi \colon \Mod_K \to \Vec_k$ that is exact, faithful, cocontinuous, and $k$-linear.
\end{definition}

The goal of this section is to classify the fiber functors on $\Mod_K$. This is accomplished in Theorem~\ref{thm:fiber} below.

\subsection{Examples of fiber functors} \label{ss:fiber-ex}

Let $x$ be a $\GL$-generic $k$-point of $X$ and let $\fm$ be the corresponding maximal ideal of $R$. Define a functor
\begin{displaymath}
\tilde{\Phi}_x \colon \Mod_R \to \Vec_k, \qquad \tilde{\Phi}_x(M)=M/\fm M.
\end{displaymath}
Since every $R$-module is flat at $\fm$ (Corollary~\ref{cor:flat}), it follows that $\tilde{\Phi}_x$ is exact. Moreover, it is clear that $\tilde{\Phi}_x$ kills torsion $R$-modules. It follows that $\tilde{\Phi}_x$ factors through the generic category $\Mod_R^{\gen}$. Identifying this with $\Mod_K$, we thus obtain a functor
\begin{displaymath}
\Phi_x \colon \Mod_K \to \Vec_k.
\end{displaymath}
We now have:

\begin{proposition} \label{prop:fiber}
The functor $\Phi_x$ is a fiber functor (in a natural manner).
\end{proposition}

\begin{proof}
The functor $\tilde{\Phi}_x$ is clearly exact, cocontinuous, and $k$-linear, and also admits a natural symmetric monoidal structure; it follows that $\Phi_x$ inherits these properties. To complete the proof, we must show that $\Phi_x$ is faithful.

We first claim that if $M$ is a torsion-free $R$-module such that $M/\fm M=0$ then $M=0$. To see this, first suppose that $M$ is finitely generated. Then $M_{\fm}$ is free over $R_{\fm}$ (Proposition~\ref{prop:loc-free}). Thus the vanishing of $M/\fm M=M_{\fm}/\fm M_{\fm}$ implies that of $M_{\fm}$, and thus of $M$ since $M$ is torsion-free. We now treat the general case. Let $N$ be a finitely generated submodule of $M$. Since $M/N$ is flat at $\fm$ (Corollary~\ref{cor:flat}), the map $N/\fm N \to M/\fm M$ is injective, and so $N/\fm N=0$. Thus $N=0$ by the previous case. Since $N$ was arbitrary, it follows that $M=0$ as well.

Now, to prove faithfulness, it suffices to show that if $f \colon M \to N$ is a map of torsion-free $R$-modules such that the induced map $\ol{f} \colon M/\fm M \to N/\fm N$ vanishes then $f=0$. Thus let such an $f$ be given. Let $I$ be the image of $f$. Since $N/I$ is flat at $\fm$ (Corollary~\ref{cor:flat}), it follows that $I/\fm I$ is the image of $\ol{f}$, and thus vanishes. Hence $I=0$ by the previous paragraph, and so $f=0$ as well.
\end{proof}

\begin{remark}
Proposition~\ref{prop:fiber} was proven for $R=\Sym(\Sym^2(\bC^{\infty}))$ (and a specific choice of $x$) in \cite[\S 3]{sym2noeth}. Similar results were also proved in \cite[\S 6]{periplectic}, \cite[\S 5]{isomeric},  \cite[\S 5]{symc1sp}. However, these papers did not have the benefit of the shift theorem and its corollaries, such as Corollary~\ref{cor:flat}, and as a result the arguments given there are much more involved.
\end{remark}

It is possible that $\Mod_K$ does not admit a fiber functor. However, this can be fixed by passing to a finite extension:

\begin{proposition}
There exists a finite extension $k'/k$ such that, putting $K'=k' \otimes_k K$, the category $\Mod_{K'}$ admits a fiber functor.
\end{proposition}

\begin{proof}
There is a finite extension $k'/k$ such that $X$ contains a $\GL$-generic $k'$-point $x$ \cite[Theorem~8.8]{polygeom}. As we have seen (Lemma~\ref{lem:struc-6} and following discussion), $K'=k' \otimes_k K$ is then a $\GL$-field that is finitely generated over its invariant field $k'$. It follows that $\Phi_x$ is a fiber functor for $\Mod_{K'}$.
\end{proof}

\subsection{More examples of fiber functors}

Let $V$ be an infinite dimensional $k$-vector space. Recall that $X\{V\}=\Spec(R\{V\})$, where $R\{V\}$ is obtained by treating $R$ as a polynomial functor and evaluating on $V$. Suppose that $x$ is a $\GL$-generic $k$-point of $X\{V\}$, corresponding to the maximal ideal $\fm$ of $R\{V\}$. (By $\GL$-generic here, we mean there is no proper closed $\GL$-subvariety $Z$ of $X$ with $x \in Z\{V\}$.) Define a functor
\begin{displaymath}
\tilde{\Phi}_{V,x} \colon \Mod_R \to \Mod_k, \qquad
\tilde{\Phi}_{V,x}(M) = M\{V\}/\fm M\{V\}.
\end{displaymath}
Once again, this functor is exact and kills torsion modules, and thus induces a functor
\begin{displaymath}
\Phi_{V,x} \colon \Mod_K \to \Mod_k.
\end{displaymath}
The same argument as in Proposition~\ref{prop:fiber} shows that it too is a fiber functor. If $V=k \otimes \bV$ then $\Phi_{V,x}$ is the funtor $\Phi_x$ introduced above. We note that $\Phi_{V,x}(K \otimes \bV)=V$, and so $\Phi_{V,x}$ and $\Phi_{V',x'}$ can only be isomorphic if $\dim{V}=\dim{V'}$ (as cardinal numbers). In particular, if $\dim(V) \ne \dim(\bV)$ then $\Phi_{V,x}$ will not be isomorphic to a fiber functor of the form $\Phi_{x'}$.

\subsection{The main theorem}

In the remainder of this section, a pair $(V,x)$ will always stand for an infinite dimensional $k$-vector space $V$ and a $\GL$-generic $k$-point $x$ of $X\{V\}$. If $(V',x')$ is a second such pair, then an \defn{isomorphism} $(V,x) \to (V',x')$ is a linear isomorphism $V \to V'$ such that the induced map $X\{V'\} \to X\{V\}$ carries $x'$ to $x$. The following theorem classifies fiber functors:

\begin{theorem} \label{thm:fiber}
We have the following:
\begin{enumerate}
\item Any fiber functor on $\Mod_K$ is isomorphic to one of the form $\Phi_{V,x}$.
\item Given two pairs $(V,x)$ and $(V',x')$, we have a natural bijection
\begin{displaymath}
\Isom((V,x),(V',x')) = \Isom(\Phi_{V,x},\Phi_{V',x'}).
\end{displaymath}
These bijections are compatible with composition of isomorphisms.
\end{enumerate}
\end{theorem}

The theorem is proved in \S \ref{ss:fiber-proof} below. We make a few remarks here.

\begin{remark}
The theorem can be stated more concisely as: the groupoid of fiber functors on $\Mod_K$ is equivalent to the groupoid of pairs $(V,x)$.
\end{remark}

\begin{remark}
It follows from the theorem that the automorphism group of the fiber functor $\Phi_{V,x}$ is the stabilizer of $x$ in the group $\Aut_k(V)$. In most cases, this group will be finite, and so $\Mod_K$ cannot be recovered as its representation category. This issue is addressed in \S \ref{s:genstab} and \S \ref{s:glstab} by introducing the notion of ``generalized stabilizers.''
\end{remark}

\begin{remark}
Given $K$, there are potentially many choices of $X$. The theorem implies that any two choices of $X$ have the same set of $\GL$-generic points (up to natural bijection). In fact, this can be seen directly. Suppose $R'$ is a second $\GL$-algebra that is finitely $\GL$-generated over $k$ and has $\Frac(R')=K$, and let $X'=\Spec(R')$. One can show that $X$ and $X'$ are birational, in the sense that there are open $\GL$-subsets $U \subset X$ and $U' \subset X'$ and an isomorphism $i \colon U \to U'$ of $\GL$-varieties. Every $\GL$-generic point of $X$ is contained in $U$, and similarly every $\GL$-generic point of $X'$ is contained in $U'$ (see \cite[Proposition~3.4]{polygeom}). Clearly, these points are mapped bijectively to one another via $i$.
\end{remark}

\subsection{Proof of Theorem~\ref{thm:fiber}} \label{ss:fiber-proof}

Let $V$ be a vector space, let $x$ be a $k$-point of $X\{V\}$, and let $\fm$ be the corresponding maximal ideal of $R\{V\}$. We define $\tilde{\Phi}_{V,x}$ as above; that is, for an $R$-module $M$, we put
\begin{displaymath}
\tilde{\Phi}_{V,x}(M)=M\{V\}/\fm M\{V\}.
\end{displaymath}
Previously, we had only used this when $V$ is infinite dimensional and $x$ is $\GL$-generic, but we now consider it more generally.

\begin{lemma} \label{lem:fiber-1}
Let $V$ and $x$ be as above. Suppose that there is a fiber functor $\Phi$ on $\Mod_K$ such that $\tilde{\Phi}_{V,x}(M)=\Phi(K \otimes_R M)$. Then $V$ is infinite dimensional and $x$ is $\GL$-generic.
\end{lemma}

\begin{proof}
Suppose, by way of contradiction, that $x$ is not $\GL$-generic (which is automatic if $V$ is finite dimensional). There is then a non-zero $\GL$-ideal $I$ of $R$ such that $x$ belongs to the vanishing locus of $I\{V\}$. Then $\tilde{\Phi}_{V,x}(R/I)=R\{V\}/\fm$ is non-zero. On the other hand, $K \otimes_R R/I=0$, and so $\Phi(K \otimes_R R/I)=0$. This is a contradiction, which completes the proof.
\end{proof}

\begin{lemma} \label{lem:fiber-2}
Let $\Phi$ be a fiber functor on $\Mod_K$. Then $\Phi$ is isomorphic to some $\Phi_{V,x}$.
\end{lemma}

\begin{proof}
Let $V=\Phi(\bV \otimes K)$. Suppose $U$ is a polynomial representation of $\GL$. Then we have $U \otimes K = U(\bV \otimes K)$, where on the right side we treat $U$ as a polynomial functor and apply it to the object $\bV \otimes K$ of $\Mod_K$. We thus find
\begin{displaymath}
\Phi(U \otimes K)=\Phi(U(\bV \otimes K))=U(\Phi(\bV \otimes K))=U\{V\},
\end{displaymath}
where in the second step we used that $\Phi$ commutes with the action of poylnomial functors, as $\Phi$ is symmetric monoidal.

We have a natural surjective map $\alpha \colon R \otimes_k K \to K$ of algebra objects in $\Mod_K$, given by multiplication. Applying $\Phi$, and appealing to the above, this yields a surjective $k$-algebra homomorphism map $\beta \colon R\{V\} \to k$. Let $\fm=\ker(\beta)$, a maximal ideal of $R\{V\}$, and let $x \in X\{V\}$ is the associated point.

Now, let $M$ be an $R$-module. Choose a presentation
\begin{displaymath}
U_1 \otimes_k R \to U_0 \otimes_k R \to M \to 0
\end{displaymath}
where $U_0$ and $U_1$ are polynomial representations. We obtain a commutative diagram
\begin{displaymath}
\xymatrix@C=4em{
U_1 \otimes_k R \otimes_k K \ar[r] \ar[d]^{\id \otimes \alpha} & U_0 \otimes_k R \otimes_k K \ar[r] \ar[d]^{\id \otimes \alpha} & M \otimes_k K \ar[r] \ar[d] & 0 \\
U_1 \otimes_k K \ar[r] & U_0 \otimes_k K \ar[r] & M \otimes_R K \ar[r] & 0 }
\end{displaymath}
with exact rows. Applying $\Phi$, we obtain a commutative diagram
\begin{displaymath}
\xymatrix@C=4em{
U_1\{V\} \otimes_k R\{V\} \ar[r] \ar[d]^{\id \otimes \beta} & U_0\{V\} \otimes_k R\{V\} \ar[r] \ar[d]^{\id \otimes \beta} & M\{V\} \ar[r] \ar[d] & 0 \\
U_1\{V\} \ar[r] & U_0\{V\} \ar[r] & \Phi(M \otimes_R K) \ar[r] & 0 }
\end{displaymath}
It follows that the right vertical map induces an isomorphism
\begin{displaymath}
\tilde{\Phi}_{V,x}(M)=M\{V\}/\fm M\{V\} \to \Phi(M \otimes_R K).
\end{displaymath}
By Lemma~\ref{lem:fiber-1}, we see that $V$ is infinite dimensional and $x$ is $\GL$-generic. The above isomorphism thus induces an isomorphism $\Phi \cong \Phi_{V,x}$.
\end{proof}

\begin{lemma}
Let $(V,x)$ and $(V',x')$ be given. Then we have a natural bijection
\begin{displaymath}
\Isom((V,x),(V',x')) = \Isom(\Phi_{V,x}, \Phi_{V',x'})
\end{displaymath}
that is compatible with composition of isomorphisms.
\end{lemma}

\begin{proof}
We first construct a map
\begin{displaymath}
\alpha \colon \Isom((V,x),(V',x')) \to \Isom(\Phi_{V,x}, \Phi_{V',x'}).
\end{displaymath}
Thus let $f \colon V \to V'$ be a $k$-linear isomorphism such that the induced map $X\{V'\} \to X\{V\}$ takes $x'$ to $x$. It follows that under the induced ring homomorphism $R\{V\} \to R\{V'\}$ the ideal $\fm'$ contracts to the ideal $\fm$. Let $M$ be an $R$-module. Then $f$ induces an isomorphism $M\{V\} \to M\{V'\}$, which further induces an isomorphism on the quotients by $\fm$ and $\fm'$. This yields an isomorphism $\tilde{\Phi}_{V,x} \cong \tilde{\Phi}_{V',x'}$ which, in turn, leads to an isomorphism $g \colon \Phi_{V,x} \cong \Phi_{V',x'}$. We define $\alpha(f)=g$.

We now define a map
\begin{displaymath}
\beta \colon \Isom(\Phi_{V,x}, \Phi_{V',x'}) \to \Isom((V,x),(V',x')).
\end{displaymath}
Let $g \colon \Phi_{V,x} \to \Phi_{V',x'}$ be an isomorphism of fiber functors. As $\Phi_{V,x}(\bV \otimes_k K) = V$, and similarly for $\Phi_{V',x'}$, we see that $g$ induces a $k$-linear isomorphism $f \colon V \to V'$. Let $I$ be the kernel of the map $R \otimes_k K \to K$ in $\Mod_K$. As we have seen, $\Phi_{V,x}=\fm$, and similarly for $\Phi_{V',x'}$. We thus see that under the ring isomorphism $R\{V\} \to R\{V'\}$ induces by $f$, the ideal $\fm$ is taken to $\fm'$. Thus $f$ defines an isomorphism $(V,x) \to (V',x')$. We put $\beta(g)=f$.

We leave to the reader the verification that $\alpha$ and $\beta$ are mutually inverse, and that these bijections are compatible with composition of isomorphisms.
\end{proof}

\section{Germinal subgroups and their representations} \label{s:genstab}

In this section, we introduce germinal subgroups (\S \ref{ss:asub}), their representation theory (\S \ref{ss:arep}), and generalized stabilizers (\S \ref{ss:genstab}) in the abstract. We also describe a general procedure for construction representations of generalized stabilizers (\S \ref{ss:eqbun}). This theory is applied in the next section when we study generalized stabilizers on $\GL$-varieties.

\subsection{Germinal subgroups} \label{ss:asub}

Fix a group $G$. The following definition introduces the main concept studied in this section:

\begin{definition} \label{def:asub}
A \defn{germinal subgroup} of $G$ is a family $\Gamma=\{\Gamma(i)\}_{i \in I}$, where $I$ is a directed set and each $\Gamma(i)$ is a subset of $G$, satisfying the following conditions:
\begin{enumerate}
\item If $i \le j$ then $\Gamma(j) \subset \Gamma(i)$.
\item Each $\Gamma(i)$ contains the identity element.
\item Given $g \in \Gamma(i)$ there is some $j \in I$ such that $\Gamma(j) g \subset \Gamma(i)$. \qedhere
\end{enumerate}
\end{definition}

The generalized stabilizer of a point on a $\GL$-variety will be a germinal subgroup. In this case, the intersection of the sets $\Gamma(i)$ will be the usual stabilizer, which is typically ``too small.'' Each of the sets $\Gamma(i)$, on the other hand, is ``too big.'' One can think of the germinal subgroup $\Gamma$ as a kind of filter on $G$ that is attempting to pick out a hypothetical subset that is bigger than the intersection but smaller than each $\Gamma(i)$. As this picture suggests, one should always be allowed to pass to a cofinal subset of $I$ when working in the setting of germinal subgroups.

\subsection{Representations} \label{ss:arep}

We fix a germinal subgroup $\Gamma=\{\Gamma(i)\}_{i \in I}$ of $G$ for \S \ref{ss:arep}.

\begin{definition} \label{def:prerep}
A \defn{pre-representation} of $\Gamma$ over a field $k$ consists of a $k$-vector space $V$ and a linear function
\begin{displaymath}
V \to \varinjlim_{i \in I} \Fun(\Gamma(i), V).
\end{displaymath}
Suppose that $V$ and $W$ are pre-representationss of $\Gamma$ over $k$. A \defn{map} of pre-representations is a $k$-linear map $V \to W$ such that the obvious diagram commutes.
\end{definition}

Suppose $V$ is a pre-representation. Given $v \in V$, its image in $\varinjlim_{i \in I} \Fun(\Gamma(i), V)$ is represented by a function $\Gamma(i) \to V$ for some $i$. Given an element $g$ of this $\Gamma(i)$, we denote its image in $V$ under this function by $gv$. We thus think of a pre-representation as a kind of partially defined action map $G \times V \dashrightarrow V$.

\begin{definition}
A \defn{representation} of $\Gamma$ is a pre-representation $V$ such that the following two conditions hold:
\begin{itemize}
\item We have $1v=v$ for all $v \in V$.
\item Given $v\in V$ there exists $i \in I$ such that for each $g \in \Gamma(i)$ there exists some $j\ge i$ such that $h(gv)=(hg)v$ for all $h \in \Gamma(j)$.
\end{itemize}
A \defn{map} of representations is simply a map of pre-representations. We let $\Rep(\Gamma)$ be the category of representations of $\Gamma$ over $k$.
\end{definition}

We make a number of remarks concerning this definition.
\begin{itemize}
\item Let $V$ be a representation of $\Gamma$ and let $W$ be a subspace of $V$. Then $W$ is a subrepresentation of $V$ if and only if for every $w \in W$ there exists $i \in I$ such that $\Gamma(i) w \subset W$.
\item Let $V$ and $W$ be representations of $\Gamma$ and let $f \colon V \to W$ be a linear map. Then $f$ is a map of representations if and only if for each $v \in V$ there exists $i \in I$ such that $f(gv)=gf(v)$ for all $g \in \Gamma(i)$.
\item Let $V$ be a representation of $G$. Then $V$ naturally carries the structure of a $\Gamma$-representation. A similar comments applies to maps of representations. We thus have a restriction functor $\Rep(G) \to \Rep(\Gamma)$.
\item The category $\Rep(\Gamma)$ is abelian. Kernels, cokernels, images, (arbitrary) direct sums, and direct limits are given in the usual manner on the underlying vector spaces. It follows that axiom (AB5) holds.
\item Let $V$ be a representation of $\Gamma$. Extend the partially defined action map to a function $G \times V \to V$ in any manner. This gives $V$ the structure of a module over the non-commutative polynomial $R$ ring with variables indexed by $g$. Let $\kappa$ be the dimension of $R$ as a $k$-vector space. One easily sees that any $R$-submodule of $V$ is a $\Gamma$-subrepresentation. It follows that every $v \in V$ is contained in a $\Gamma$-subrepresentation of dimension at most $\kappa$, namely, $Rv$. Thus, taking one $\Gamma$ representation from each isomorphism class of representations of dimension at most $\kappa$, one obtains a generating set for $\Rep(\Gamma)$. It follows that $\Rep(\Gamma)$ is a Grothendieck abelian category. In particular, it is complete.
\item From the above, we see that $\Rep(\Gamma)$ has arbitrary products. These are not necessarily computed in the usual manner on the underlying vector space.
\item Similarly, we see that there is a notion of intersection for an arbitrary family of subrepresentations of a $\Gamma$-representation. This intersection may not coincide with the usual intersection of vector subspaces.
\item Let $V$ and $W$ be representations of $\Gamma$. We give the vector space $V \otimes W$ the structure of a representation in the usual manner: that is, we define
\begin{displaymath}
g \cdot \big( \sum_{i=1}^n v_i \otimes w_i \big) = \sum_{i=1}^n gv_i \otimes gw_i,
\end{displaymath}
provided $gv_i$ and $gw_i$ are defined for all $i$. One easily verifies that this is indeed a representation. This construction endows $\Rep(\Gamma)$ with a symmetric monoial structure.
\end{itemize}

\subsection{Weak subrepresentations}

Given any vector space $V$, the dual space $V^*$ carries a natural topology: namely, a sequence (or net) $\{\lambda_j\}_{j \in J}$ in $V^*$ converges to $\lambda$ if for every vector $v \in V$ there is some $j_0 \in J$ such that $\lambda_j(v)=\lambda(v)$ for all $j \ge j_0$. We call this the \defn{$\Pi$-topology}. For a subspace $W$ of $V$, we let $W^{\perp} \subset V^*$ be its annihilator, i.e., the set of functionals $\lambda \in V^*$ such that $\lambda(w)=0$ for all $w \in W$. One easily sees that $W^{\perp}$ is $\Pi$-closed, and that $W \mapsto W^{\perp}$ is a bijection between subspaces of $V$ and closed subspaces of $V^*$.

Let $V$ be a representation of $G$ and let $W$ be a subspace. A \emph{$\Gamma$-sequence} is a sequence $\{g_j\}_{j \in J}$ in $G$, indexed by some directed set $J$, such that for each $i \in I$ there exists $j_0 \in J$ such that $g_j \in \Gamma(i)$ for all $j \ge j_0$. We say that $W$ is a \defn{weak $\Gamma$-subrepresentation} of $V$ if it satisfies the following condition: given $\lambda \in W^{\perp}$ and a $\Gamma$-sequence $\{g_j\}$ such that $g_j^{-1} \lambda$ converges in $V^*$ to an element $\mu$, we have $\mu \in W^{\perp}$.

\begin{proposition} \label{prop:weak-sub}
Let $V$ be a representation of $G$ and let $W$ be $\Gamma$-subrepresentation of $V$. Then $W$ is a weak $\Gamma$-subrepresentation of $V$.
\end{proposition}

\begin{proof}
Let $\lambda \in W^{\perp}$ and let $\{g_j\}_{j \in J}$ be a $\Gamma$-sequence such that $g_j^{-1} \lambda$ converges in $V^*$ to some element $\mu$. Let $w \in W$. Since $W$ is a $\Gamma$-subrepresentation, there exists $i \in I$ such that $gw \in W$ for all $g \in \Gamma(i)$. Since $g_j^{-1} \lambda$ converges to $\mu$ there is some $j_0 \in J$ such that $\mu(w)=\lambda(g_jw)$ for all $j \ge j_0$. Let $j_1 \ge j_0$ be such that $g_j \in \Gamma(i)$ for all $j \ge j_1$. Then for $j \ge j_1$ we have $\mu(w)=\lambda(g_jw)=0$ since $g_iw \in W$ and $\lambda$ vanishes on $V$. Thus $\mu$ vanishes on $W$, and so $\mu \in W^{\perp}$. This shows that $W$ is a weak subrepresentation.
\end{proof}

\subsection{Generalized stabilizers} \label{ss:genstab}

Let $I$ be a directed set and let $\{X_i\}_{i \in I}$ be an inverse system of sets; for $i \le j$, let $\pi_{j,i} \colon X_j \to X_i$ be the transistion map. Let $X$ be the inverse limit of the system. For $i \in I$, we let $\pi_i \colon X \to X_i$ be the natural map. We suppose that a group $G$ acts on $X$, and that the action satisfies the following condition: given $g \in G$ and $i \in I$ there exists $j \in I$ such that $\pi_i \circ g$ factors through $\pi_j$; in other words, one can complete the following commutative diagram:
\begin{displaymath}
\xymatrix@C=3em{
X \ar[d]_{\pi_j} \ar[r]^g & X \ar[d]^{\pi_i} \\
X_j \ar@{..>}[r] & X_i }
\end{displaymath}
Equivalently, this means that each $g \in G$ acts uniformly continuously on $X$, when $X$ is endowed with the inverse limit uniform structure (and each $X_i$ with the discrete uniform structure).

We now come to a fundamental definition:

\begin{definition}
Let $x \in X$. For $i \in I$, let $\Gamma_x(i)$ be the set of elements $g \in G$ such that $\pi_i(g^{-1}x)=\pi_i(x)$. The \defn{generalized stabilizer} of $x$ is the system $\Gamma_x=\{\Gamma_x(i)\}_{i \in I}$.
\end{definition}

\begin{proposition}
The generalized stabilizer $\Gamma_x$ is a germinal subgroup of $G$.
\end{proposition}

\begin{proof}
We verify the three conditions of Definition~\ref{def:asub}. It is clear that $1 \in \Gamma_x(i)$ for all $i$, which verifies condition (a). If $g \in \Gamma_x(j)$ and $i \le j$ then taking the given identity $\pi_j(g^{-1}x)=\pi_j(x)$ and applying the transition map $\pi_{j,i}$, we find that $\pi_i(g^{-1}x)=\pi_i(x)$, and so $g \in \Gamma_x(i)$. This shows that $\Gamma_x(j) \subset \Gamma_x(i)$, which verifies condition (b).

Finally, we come to condition (c). Suppose $g \in \Gamma_i$. Let $j$ be such that we have a factorization $\pi_i \circ g^{-1} = \phi \circ \pi_j$ for some $\phi \colon X_j \to X_i$. Suppose $h \in \Gamma_x(j)$. Then $\pi_j(h^{-1}x)=\pi_j(x)$. Applying $\phi$, we find $\pi_i(g^{-1} h^{-1} x)=\pi_i(x)$, which shows that $hg \in \Gamma_x(i)$. Thus $\Gamma_x(j) g \subset \Gamma_x(i)$, as required.
\end{proof}

\begin{proposition}
The intersection $\bigcap_{i \in I} \Gamma_x(i)$ is the usual stabilizer of $x$, i.e., the set of all $g \in G$ such that $gx=x$.
\end{proposition}

\begin{proof}
It is clear that if $g$ stabilizes $x$ then $g \in \Gamma_x(i)$ for all $i$. Conversely, if $g \in \Gamma_x(i)$ for all $i$ then we have $\pi_i(g^{-1} x)=\pi_i(x)$ for all $i$, and so $g^{-1} x=x$, which shows that $g$ stabilizes $x$.
\end{proof}

\subsection{Representations from equivariant bundles} \label{ss:eqbun}

Maintain the notation from \S \ref{ss:genstab}. We now describe how to produce representations of $\Gamma_x$ from certain kinds of equivariant vector bundles on $X$. This discussion is included simply to offer some intuition for germinal subgroups, and is not used in what follows.

For each $i \in I$, let $E_i$ be a vector bundle on $X_i$; since $X_i$ is discrete, this simply amounts to giving a vector space $E_i(x)$ for each $x \in X_i$. To keep this discussion less technical, we assume that each $E_i(x)$ is finite dimensional. Suppose that the dual bundles $\{E_i^*\}_{i \in I}$ have the structure of an inverse system of vector bundles, and let $E^*$ be the inverse limit, which is a vector bundle on $X$ (in a loose sense; it may not be locally trivial). For a point $x=\{x_i\}_{i \in I}$ of $X$, the fiber $E^*(x)$ is the inverse limit of the vector spaces $E^*_i(x_i)$. Define $E(x)$ to be the corresponding direct limit; note that $E^*(x)$ is the dual space of $E(x)$. We say that $x \in X$ is \defn{good} if there exists $i_0 \in I$ such that the transition map $E_i(x_i) \to E_j(x_j)$ is injective for all $i_0 \le i \le j$.

Suppose now that $E^*$ is endowed with a $G$-equivariant structure. Thus for $g \in G$ and $x \in X$ we have linear isomorphisms $g \colon E^*(x) \to E^*(gx)$ and $g \colon E(x) \to E(gx)$ that satisfy the cocycle conditions. As in the previous section, we assume the map $g \colon E^* \to g^*(E^*)$ is uniformly continuous. Let $x=\{x_i\}_{i \in I}$ be a good point. We claim that $E(x)$ is naturally a representation of the generalized stabilizer $\Gamma_x$. Indeed, suppose $g \in \Gamma_x(i)$, so that $\pi_i(g^{-1}x)=\pi_i(x)$. We have a (likely non-commutative) diagram
\begin{displaymath}
\xymatrix{
& E_i(x_i) \ar[ld]_{\alpha} \ar[rd]^{\beta} \\
E(x) && E(g^{-1}x) \ar[ll]_g }
\end{displaymath}
Assuming $i$ is large enough, $\alpha$ is an inclusion. For $x=\alpha(y)$, we define $gx$ to be the element $g \beta(y)$. One easily verifies that this is independent of $i$, and defines the structure of a $\Gamma$-representation on $E(x)$.

\section{Generalized stabilizers on \texorpdfstring{$\GL$}{GL}-varieties} \label{s:glstab}

In this final section, we study the generalized stabilizer $\Gamma_x$ of a point $x$ on a $\GL$-variety $X$. Our main result provides an equivalence between the category of polynomial representations of $\Gamma_x$ and the category of $K$-modules when $x$ is a $\GL$-generic point on $X=\bA^{\ulambda}$. This yields the statements in \S \ref{ss:repgs}, as the corresponding statements for $\Mod_K$ have already been established.

\subsection{Generalized stabilizers on \texorpdfstring{$\GL$}{GL}-varieties}

Let $X=\Spec(R)$ be an irreducible affine $\GL$-variety over the field $k$. Let $R_n=R\{k^n\}$ be the ring obtained by evaluating $R$ on $k^n$ and let $X_n=\Spec(R_n)$, a finite dimensional variety over $k$. Then $X(k)$ is the inverse limit of the $X_n(k)$ in the category of sets. Let $\pi_n \colon X(k) \to X_n(k)$ be the natural map. Given $g \in \GL$ and $x \in X(k)$, we see that $\pi_n(gx)$ can be obtained from the image of $x$ in $X\{g^{-1}k^n\}$ by applying $g$. Thus if $m \ge n$ is such that $g^{-1}k^n \subset k^m$, then one can recover $\pi_n(gx)$ from $\pi_m(x)$. This shows that the action of $g$ is uniformly continuous, as described in \S \ref{ss:genstab}.

Fix a point $x \in X(k)$. Let $\Gamma_x$ be its generalized stabilizer for the action of $\GL$ on $X(k)$. Thus $\Gamma_x(n)$ is the set of elements $g \in \GL$ such that $g^{-1}x$ and $x$ have the same image in $X_n(k)$. Letting $\fm \subset R$ be the defining ideal of $x$, we see that $\Gamma_x(n)$ can also be described as the set of elements $g \in \GL$ such that $g^{-1} \fm \cap R_n=\fm \cap R_n$.

We say that a representation $V$ of $\Gamma_x$ is \defn{polynomial} if there is a polynomial representation $W$ of $\GL$ such that $V$ is isomorphic to a subquotient of $W$ (regarded as a representation of $\Gamma_x$). We write $\Rep^{\pol}(\Gamma_x)$ for the category of polynomial representations of $\Gamma_x$. It is a Grothendieck abelian category that is closed under tensor products.

\begin{remark}
One can also define a notion of algebraic representation of $\Gamma_x$ by using restrictions of algebraic representations of $\GL$ (as defined in, e.g., \cite[\S 3.1.1]{infrank}). In many cases, polynomial and algebraic representations coincide. We therefore confine our attention to the polynomial case.
\end{remark}

\subsection{From modules to representations}

Maintain the above setup. The following proposition is the key result that justifies our definitions:

\begin{proposition} \label{prop:Gamma-linear}
Let $V$ and $W$ be polynomial representations of $\GL$ and let $\phi \colon R \otimes V \to R \otimes W$ be a map of $R$-modules. Then the linear map $\phi_x \colon V \to W$ obtained by reducing $\phi$ modulo $\fm$ is a map of $\Gamma_x$-representations.
\end{proposition}

\begin{proof}
Let $v \in V$ be given. Let $n$ be such that $v$ is invariant under $G(n)$. We claim that $\phi_x(gv)=g \phi_x(v)$ for $g \in \Gamma_x(n)$, which will complete the proof. Thus let $g \in \Gamma_x(n)$ be given. Write $\phi(1 \otimes v)=\sum_{i=1}^r f_i \otimes w_i$ with $f_i \in R$ and $w_i \in W$. Then we have
\begin{displaymath}
\phi_x(gv)=\sum_{i=1}^r f_i(g^{-1}x) gw_i, \qquad
g\phi_x(v)=\sum_{i=1}^r f_i(x) gw_i,
\end{displaymath}
so it is enough to show that $f_i(g^{-1}x)=f_i(x)$ for each $1 \le i \le r$. Since $v$ is $G(n)$-invariant, so is $f_i$; in other words, $f_i \in R_n$. We thus see that $f_i-f_i(x)$ belongs to $\fm \cap R_n$. By definition of $\Gamma_x$, we have $g^{-1}\fm \cap R_n=\fm \cap R_n$, and so $f_i-f_i(x)$ belongs to $g^{-1}\fm$. This exactly means that $f_i-f_i(x)$ vanishes at $g^{-1}x$, i.e., $f_i(g^{-1}x)=f_i(x)$. This verifies the claim.
\end{proof}

We now suppose that $x$ is $\GL$-generic; if it is not, one can simply replace $X$ with the orbit closure of $x$. The following proposition is our main construction of $\Gamma_x$-representations:

\begin{proposition} \label{prop:Psi}
There exists a unique right exact functor
\begin{displaymath}
\tilde{\Psi}_x \colon \Mod_R \to \Rep^{\pol}(\Gamma_x)
\end{displaymath}
satisfying the following two conditions:
\begin{enumerate}
\item We have $\tilde{\Psi}_x(M)=M/\fm M$ as vector spaces (and similarly for morphisms).
\item If $V$ is a polynomial representation then the $\Gamma_x$-action on $\tilde{\Psi}_x(R \otimes V) \cong V$ is the restriction of the $\GL$ action.
\end{enumerate}
The functor $\tilde{\Psi}_x$ is exact and kills the torsion category, and thus induces a functor
\begin{displaymath}
\Psi_x \colon \Mod_K \to \Rep^{\pol}(\Gamma_x).
\end{displaymath}
The functor $\Psi_x$ is exact, cocontinuous, faithful, $k$-linear, and naturally symmetric monoidal.
\end{proposition}

\begin{proof}
Let $M$ be an $R$-module. Choose a presentation
\begin{displaymath}
\xymatrix@C=3em{
R \otimes V \ar[r]^{\phi} & R \otimes W \ar[r] & M \ar[r] & 0 }
\end{displaymath}
where $V$ and $W$ are polynomial representations. Applying $- \otimes_R R/\fm$, we obtain a sequence
\begin{displaymath}
\xymatrix@C=3em{
V \ar[r]^{\phi_x} & W \ar[r] & M/\fm M \ar[r] & 0. }
\end{displaymath}
By Proposition~\ref{prop:Gamma-linear}, the first map is one of $\Gamma_x$-representations. It follows that $M/\fm M$ inherits the structure of a $\Gamma_x$-representation, which is easily seen to be independent of the choice of presentation. This representation is polynomial since it is a quotient of $W$. One easily sees that this construction defines a right-exact functor
\begin{displaymath}
\tilde{\Psi}_x \colon \Mod_R \to \Rep^{\pol}(\Gamma_x), \qquad M \mapsto M/\fm M.
\end{displaymath}
It is clear that (a) and (b) hold. The uniqueness of $\tilde{\Psi}_x$ follows from the fact that it is right-exact and determined on the category of projective $R$-modules by (a) and (b).

Since $x$ is $\GL$-generic, $M$ is flat at $\fm$ (Corollary~\ref{cor:flat}), and so $\tilde{\Psi}_x$ is exact. It is clear that $\tilde{\Psi}_x$ kills the torsion subcategory. It thus factors through the generic category, which is equivalent to $\Mod_K$. We therefore obtain a functor $\Psi_x$ as in the statement of the proposition. Of course, ignoring the representation structure, $\Psi_x$ is just the fiber functor $\Phi_x$ we constructed in \S \ref{ss:fiber-ex}. In other words, the diagram
\begin{displaymath}
\xymatrix@C=3em{
\Mod_K \ar[r]^-{\Psi_x} \ar[rd]_{\Phi_x} & \Rep^{\pol}(\Gamma_x) \ar[d] \\
& \Vec_k }
\end{displaymath}
commutes, where the vertical arrow is the forgetful functor. It follows that $\Psi_x$ is exact, cocontinuous, faithful, and $k$-linear; moreover, one easily sees that the symmetric monoidal structure on $\Phi_x$ respects the $\Gamma_x$-structure, and so $\Psi_x$ is naturally symmetric monoidal as well.
\end{proof}

We expect that $\Psi_x$ is an equivalence in general. In the remainder of this section, we prove this when $K$ is a rational $\GL$-field (Theorem~\ref{thm:genstab}) and $k$ is algebraically clsoed.

\begin{remark}
The above construction is essentially a special case of the one from \S \ref{ss:eqbun}, as we can regard $\Spec(\Sym(M))$ as a vector bundle (loosely interpreted) over $X$.
\end{remark}

\begin{remark}
Let $M$ be a submodule of $R^{\oplus \ulambda}$, and let $V=M/\fm M \subset k^{\oplus \ulambda}$. It is easy to see that $V$ is a weak subrepresentation of $k^{\oplus \ulambda}$. Indeed, let $\cE=\Spec(\Sym(R^{\oplus \ulambda}/M))$, a closed $\GL$-subscheme of the vector bundle $X \times (k^{\oplus \ulambda})^*$, and $\cE(x)=V^{\perp}$. Suppose $\alpha \in \cE(x)$ and $\{g_i\}$ is a $\Gamma_x$-sequence such that $g_i \alpha$ converges to $\beta$ in $(k^{\oplus \ulambda})^*$. Since $g_ix$ converges to $x$, it follows that $g_i(x,\alpha)$ converges in $X \times (k^{\oplus \ulambda})^*$ to $(x,\beta)$. Since each $g_i(x,\alpha)$ belongs to $\cE$ and $\cE$ is closed, we see that $\beta \in \cE(x)$. This verifies the claim.

We had originally defined a $\Gamma_x$-representation to be a pair $(V, k^{\oplus \ulambda})$ consisting of a polynomial representation $k^{\oplus \ulambda}$ and a weak subrepresentation $V$. This can be made to work, thanks to the above proposition. However, it is not a good definition since we really just want the space $V$; the ambient representation $k^{\oplus \ulambda}$ is extrinsic. (Also, it is not immediately clear that this definition yields an abelian category.) It took some time for us to realize that the data intrinsic to $V$ is that of a $\Gamma_x$-pre-representation, as in Definition~\ref{def:prerep}.
\end{remark}

\subsection{From representations to modules} \label{ss:rep-to-mod}

We assume for the remainder of \S \ref{s:glstab} that $k$ is algebraically closed. Fix a pure tuple $\usigma$, put $R=\Sym(k^{\oplus \usigma})$, put $X=\Spec(R)=\bA^{\usigma}$, and let $K=\Frac(R)$. Fix a $\GL$-generic $k$-point $x$ of $X$, and let $\fm \subset R$ be its defining ideal. The goal of this subsection is to prove the following proposition, which is the key to the proof of Theorem~\ref{thm:genstab}.

\begin{proposition} \label{prop:weaksub}
Let $\umu$ be a tuple and let $V$ be a subspace of $k^{\oplus \umu}$. The following are equivalent:
\begin{enumerate}
\item There is an $R$-submodule $M$ of $R^{\oplus \umu}$ such that $V=M/\fm M$.
\item The space $V$ is a $\Gamma_x$-subrepresentation of $k^{\oplus \umu}$.
\item The space $V$ is a weak $\Gamma_x$-subrepresentation of $k^{\oplus \umu}$.
\end{enumerate}
\end{proposition}

We have already seen that (a) implies (b) (Proposition~\ref{prop:Psi}), and that (b) implies (c) (Proposition~\ref{prop:weak-sub}), so it suffices to prove that (c) implies (a). This will take the remainder of the subsection.

We use the theory of systems of variables from \cite[\S 9.1]{polygeom}. We say that a $k$-point of $\bA^{\ulambda}$ is \emph{degenerate} if it is not $\GL$-generic, and \emph{non-degenerate} otherwise.  For a single partition $\lambda$, the degenerate points in $\bA^{\lambda}(k)$ form a $k$-subspace \cite[Proposition~9.2]{polygeom}. A \emph{system of $\lambda$-variables} is a set of points in $\bA^{\lambda}(k)$ that forms a basis modulo the subspace of degenerate elements. A \emph{system of variables} is a choice of system of $\lambda$-variables for all $\lambda$.

\begin{lemma} \label{lem:weaksub-0}
Let $\umu$ and $\unu$ be pure tuples, let $p \in \bA^{\umu}(k)$ be $\GL$-generic, and let $E \subset \bA^{\unu}$ be the set of $k$-points $q$ such that $(q,p) \in \bA^{\unu} \times \bA^{\umu}$ is $\GL$-generic. Then $E$ is Zariski dense in $\bA^{\unu}$.
\end{lemma}

\begin{proof}
A point is non-degenerate if and only if each homogeneous piece of it is non-degenerate \cite[Proposition~9.3]{polygeom}. It thus suffices to prove the lemma when $\umu$ and $\unu$ are composed of partitions of some constant size $d$. First suppose that $d=1$. Then a point is non-degenerate if its components are linearly independent. We can clearly choose $q$ such that the components of $(q,p)$ are independent while at the same time realizing arbitary values at finitely many coordinates of $q$. Since any non-zero function $f$ on $\bA^{\unu}$ uses only finitely many coordinates, it follows that we can choose $q \in E$ such that $f(q) \ne 0$. Thus $E$ is Zariski dense.

The case when $d>1$ is similar. The set $E$ is non-empty: we can choose a system of variables that includes the components of $p$, and then take the components of $q$ to be other elements from the system. Let $q \in E$. Then we can find a degenerate $k$-point $r$ of $\bA^{\unu}$ realizing arbitary values at finitely many coordiantes. It follows that $q+r \in E$ also realizes arbitrary values at these coordinates, and so again $E$ is Zariski dense.
\end{proof}

\begin{lemma} \label{lem:weaksub-4}
Let $\umu$ be a tuple and let $p$ be a $k$-point of $\bA^{\umu}$. Then there exists a pure tuple $\unu$, a $k$-point $q$ of $\bA^{\unu}$ such that $(q,x) \in \bA^{\unu} \times X$ is $\GL$-generic, and a map of $\GL$-varieties $f \colon \bA^{\unu} \times X \to \bA^{\umu} \times X$ over $X$ such that $f(q,x)=(y,x)$.
\end{lemma}

\begin{proof}
Write $\usigma=[\sigma_1, \ldots, \sigma_r]$ and let $x=(x_1, \ldots, x_r)$ be the components of $x$. Pick a system of variables including $x_1, \ldots, x_r$. By \cite[Theorem~9.5]{polygeom}, there exists a pure tuple $\ulambda=[\lambda_1, \ldots, \lambda_s]$ and a map of $\GL$-varieties $g \colon \bA^{\ulambda} \to \bA^{\umu}$ such that $p=g(\xi_1, \ldots, \xi_s)$, where $\xi_1, \ldots, \xi_s$ are distinct elements from the system of variables. Now, after applying a permutation, we can assume that $\xi_i=x_i$ for $1 \le i \le t$ and the remaining $\xi_i$ and $x_j$ are distinct. Let $\unu=[\lambda_{t+1}, \ldots, \lambda_s]$ and $q=(\xi_{t+1}, \ldots, \xi_s) \in \bA^{\unu}$. Now, let $f$ be the composition
\begin{displaymath}
\xymatrix@C=3em{
\bA^{\unu} \times X \ar[r]^{\Delta} & \bA^{\ulambda} \times X \ar[r]^{g \times \id_X} & \bA^{\umu} \times X }
\end{displaymath}
where $\Delta$ is the diagonal map that copies the first $t$ coordinates of $X$ into those of $\bA^{\ulambda}$. Then $\Delta(q,x)=(\xi_1, \ldots, \xi_s, x_1, \ldots, x_r)$, and so $f(q,x)=(p,x)$. By construction $(q,x)$ is $\GL$-generic.
\end{proof}

Given vector spaces $V \subset U$, we let $V^{\perp}$ be the annihilator of $V$ in the dual space $U^*$.

\begin{lemma} \label{lem:weaksub-1}
Let $\umu$ be a tuple, let $V$ be a weak $\Gamma_x$-subrepresentation of $k^{\oplus \umu}$, and let $p \in V^{\perp}$. Then there exists a tuple $\unu$ and a map of $\GL$-varieties $f \colon \bA^{\unu} \times X \to \bA^{\umu} \times X$ over $X$ such that $\im(f_x)$ contains $p$ and is contained in $V^{\perp}$.
\end{lemma}

\begin{proof}
Note that $V^{\perp}$ is a subspace of $(k^{\oplus \umu})^*=\bA^{\umu}$. Applying Lemma~\ref{lem:weaksub-4}, there exists a pure tuple $\unu$, a $k$-point $q$ of $\bA^{\unu}$ such that $(q,x)$ is $\GL$-generic in $\bA^{\unu} \times X$, and a map of $\GL$-varieties $f \colon \bA^{\unu} \times X \to \bA^{\umu} \times X$ over $X$ such that $f(q,x)=(p,x)$. Thus $p \in \im(f_x)$.

Now, let $q'$ be a $k$-point of $\bA^{\unu}$ such that $(q',x)$ is $\GL$-generic. We claim that $f_x(q') \in V^{\perp}$. Since $(q,x)$ is $\GL$-generic there is a sequence $\{g_i\}_{i \ge 1}$ in $\GL$ such that $g_i (q,x)$ converges to $(q', x)$ in the $\Pi$-topology (Proposition~\ref{prop:Pi-Zariski}). We thus see that $g_i x$ converges to $x$ in the $\Pi$-topology, and so $\{g_i\}$ is a $\Gamma_x$-sequence. Applying $f$, we see that $g_i p$ converges to $f(q')$. Since $V$ is a weak subrepresentation, this implies that $f(q') \in V^{\perp}$, as claimed.

Now, let $E$ be the set of $k$-points $q' \in \bA^{\unu}$ such that $(q',x)$ is $\GL$-generic. By the previous paragraph, we see that $f_x(E) \subset V^{\perp}$. Since $E$ is Zariski dense in $\bA^{\unu}$ by Lemma~\ref{lem:weaksub-0} and $V^{\perp}$ is a Zariski closed subset of $\bA^{\umu}$, it follows that $\im(f_x) \subset V^{\perp}$, as required.
\end{proof}

\begin{lemma} \label{lem:weaksub-2}
Let $\umu$ and $\unu$ be tuples, and let $f \colon \bA^{\unu} \times X \to \bA^{\umu} \times X$ be a map of $\GL$-varieties over $X$. Then there exists a closed $\GL$-subvariety $Y$ of $\bA^{\umu} \times X$ such that the following two conditions hold:
\begin{enumerate}
\item $Y$ is defined by fiberwise linear equations, that is, $Y=\Spec(\Sym(M))$ for some $R$-module quotient $M$ of $R^{\oplus \umu}$
\item the $k$-subspace $Y_x$ of $\bA^{\umu}$ is exactly the $\Pi$-closure of the $k$-span of $\im(f_x)$.
\end{enumerate}
\end{lemma}

\begin{proof}
First suppose that $f$ is fiberwise linear. This means that $f$ is induced from a map of $R$-modules $g \colon R^{\oplus \umu} \to R^{\oplus \unu}$. Let $M$ be the image of $g$, and let $Y=\Spec(\Sym(M))$. Let $g_x \colon k^{\oplus \umu} \to k^{\oplus \unu}$ be the map obtained by reducing $g$ modulo the maximal ideal $\fm$. Since $\coker(g)$ is flat at $x$ (Corollary~\ref{cor:flat}), it follows that the image of $g_x$ is $M/\fm M$. As $f_x$ is the dual of $g_x$, we see that its image is the dual of $M/\fm M$, which is exactly $Y_x$. This completes the proof in the linear case. (In this case, taking the $\Pi$-closure is not necessary.)

We now treat the general case. The map $f$ corresponds to a map of $R$-algebras $g \colon R \otimes \Sym(k^{\oplus \umu}) \to R \otimes \Sym(k^{\oplus \unu})$. The image of $k^{\oplus \umu}$ under this map is contained in $R \otimes \Sym^{\le d}(k^{\oplus \unu})$ for some $d$, where $\Sym^{\le d}=\bigoplus_{i=0}^d \Sym^i$. The map $g$ then factors as
\begin{displaymath}
\xymatrix@C=3em{
R \otimes \Sym(k^{\oplus \umu}) \ar[r]^-{g_1} &
R \otimes \Sym(\Sym^{\le d}(k^{\oplus \unu})) \ar[r]^-{g_2} &
 R \otimes \Sym(k^{\otimes \unu}), }
\end{displaymath}
where $g_1$ is linear (i.e., induced from a map of $R$-modules). Let $f=f_1 \circ f_2$ be the corresponding factorization of $f$. Let $Y \subset \bA^{\umu} \times X$ be the subvariety provided by the linear case, applied to $f_1$. The map $f_x$ factors as
\begin{displaymath}
\xymatrix@C=3em{
(k^{\oplus \unu})^* \ar[r]^-{f_{2,x}} &
(\Sym^{\le d}(k^{\oplus \unu}))^* \ar[r]^-{f_{1,x}} &
(k^{\oplus \unu})^* }
\end{displaymath}
We know that the image of $f_{1,x}$ is exactly $Y_x$. The map $f_{2,x}$ is the canonical map, taking $a$ to $(1, a, \ldots, a^d)$. One easily sees that the $k$-span of the image of $f_{2,x}$ is $\Pi$-dense. Since $f_{1,x}$ is $\Pi$-continuous, the result follows.
\end{proof}

\begin{proof}[Proof of Proposition~\ref{prop:weaksub}]
Let $V$ be a weak $\Gamma_x$-subrepresentation of $k^{\oplus \umu}$, and fix an element $v \in V^{\perp}$. By Lemma~\ref{lem:weaksub-1}, we can find a tuple $\unu$ and a map of $\GL$-varieties $f \colon \bA^{\unu} \times X \to \bA^{\umu} \times X$ over $X$ such that $\im(f_x)$ contains $v$ and is contained in $V$. By Lemma~\ref{lem:weaksub-2}, there is an $R$-module $M=R^{\oplus \umu}/N$ such that $(M/\fm M)^*$ is the $\Pi$-closure of the span of $\im(f_x)$. Since $V^{\perp}$ is $\Pi$-closed, it follows that $(M/\fm M)^*$ is contained in $V^{\perp}$; of course, it also contains $v$. We thus see that $N/\fm N$ contains $V$ and is contained in $\ker(v)$.

Now, let $\{v_i\}_{i \in U}$ be a basis for $V^{\perp}$, and for each $i$ pick a submodule $N_i$ of $R^{\oplus \umu}$ as in the previous paragraph, so that $N_i/\fm N_i$ contains $V$ and is contained in $\ker(v_i)$. For a finite subset $I$ of $U$, let $N_I=\bigcap_{i \in I} N_i$. The $N_I$ form a descending family of submodules of $R^{\oplus \umu}$. Since $R^{\oplus \umu}$ is an artinian object in the generic category (Theorem~\ref{thm:rat-struc}(a)), it follows that there is some finite subset $J$ such that $N_J/N_I$ is torsion for all $J \subset I$. We thus have $N_J/\fm N_J=N_I/\fm N_I$ for all such $I$. It follows that $N_J/\fm N_J$ is contains $V$ and is contained in $\bigcap_{i \in U} \ker(v_i)=V$. This completes the proof.
\end{proof}

\subsection{The main theorem}

Maintain the setup from \S \ref{ss:rep-to-mod}. The following is our main theorem on representations of $\Gamma_x$:

\begin{theorem} \label{thm:genstab}
The functor $\Psi_x \colon \Mod_K \to \Rep^{\pol}(\Gamma_x)$ is an equivalence.
\end{theorem}

From the theorem, we see that all properties of $\Mod_K$ transer to $\Rep^{\pol}(\Gamma_x)$. This yields the statements of \S \ref{ss:repgs}. (We note that in the setting of \S \ref{ss:repgs}, there is no distinction between algebraic and polynomial representation.) Before proving the theorem, we require a lemma.

\begin{lemma} \label{lem:genstab}
Let $V$ be a $K$-module. Then the map
\begin{displaymath}
\alpha \colon \{ \text{$K$-submodules of $V$} \} \to \{ \text{$\Gamma_x$-subrepresentations of $\Psi_x(V)$} \}
\end{displaymath}
induced by $\Psi_x$ is an isomorphism of partially ordered sets.
\end{lemma}

\begin{proof}
We first show that $\alpha$ is injective. First suppose that $U \subset W$ are $K$-submodules of $V$ and $\alpha(U)=\alpha(W)$. Then the containment of $R$-modules $U^{\pol} \subset W^{\pol}$ induces an isomorphism modulo $\fm$. It follows that $W^{\pol}/U^{\pol}$ has vanishing fiber at $\fm$, and thus vanishes (see the proof of Proposition~\ref{prop:fiber}). Hence $U^{\pol}=W^{\pol}$, and so $U=W$. Now suppose that $U$ and $W$ are arbitary and $\alpha(U)=\alpha(W)$. Then $\alpha(U+W)=\alpha(U)+\alpha(W)=\alpha(U)$. Since $U \subset U+W$, the previous case shows that $U=U+W$, and so $W \subset U$. By symmetry, we have $U \subset W$. Thus $\alpha$ is injective.

We now see that $\alpha$ is strictly order-preserving. Indeed, let $U$ and $W$ be $K$-submodules of $V$. If $U \subset W$ then it is clear that $\alpha(U) \subset \alpha(W)$. Conversely, if $\alpha(U) \subset \alpha(W)$ then $\alpha(U+W)=\alpha(U)+\alpha(W)=\alpha(W)$, and so $U+W=W$ since $\alpha$ is injective, whence $U \subset W$.

To complete the proof, we must show that $\alpha$ is surjective. If $V=K \otimes V_0$ for a finite length polynomial representation $V_0$, then this follows from Proposition~\ref{prop:weaksub}. Suppose now that $V=K \otimes V_0$ for an arbitrary polynomial representation $V_0$. Write $V_0=\bigcup_{j \in J} V_{0,j}$ where $J$ is a directed set and $V_{0,j}$ has finite length, and put $V_j=K \otimes V_{0,j}$. Since $\Psi_x$ is cocontinuous, we have $\Psi_x(V)=\bigcup_{j \in J} \Psi_x(V_j)$. Let $E$ be a $\Gamma_x$-subrepresentation of $\Psi_x(V)$, and put $E_j=E \cap \Psi_x(V_j)$. Since $\Rep^{\pol}(\Gamma_x)$ is a Grothendieck category, we have $E=\bigcup_{j \in J} E_j$. By the finite length case, we have $E_j=\alpha(W_j)$ for a unique $K$-submodule $W_j$ of $V$. Since $\alpha$ is strictly order-preserving, it follows that $W_j \subset W_k$ if $j \le k$. Thus the $W_j$'s form a directed system. Let $W=\bigcup_{j \in J} W_j$. Again, by the cocontinuity of $\Psi_x$, we have $\alpha(W)=\bigcup_{j \in J} \alpha(W_j) = E$.

Finally, suppose that $V$ is an arbitrary $K$-module. Since $\Mod_K$ is a Grothendieck abelian category, $V$ embeds into an injective object $I$. Since $\Mod_K$ is locally noetherian (Theorem~\ref{thm:rat-struc}(a)), $I$ is a direct sum of indecomposable injectives. Thus $I$ has the form $K \otimes V_0$ for a polynomial representation $V_0$ (Theorem~\ref{thm:rat-struc}(b)). Now, suppose that $E$ is a $\Gamma_x$-subrepresentation of $\Psi_x(V)$. Since $\Psi_x(V) \subset \Psi_x(I)$, the previous paragraph shows that $E=\alpha(W)$ for some $K$-submodule $W$ of $I$. Since $\alpha$ is strictly order preserving, it follows that $W \subset V$, which completes the proof.
\end{proof}

\begin{proof}[Proof of Theorem~\ref{thm:genstab}]
We first show that $\Psi_x$ is essentially surjective. Thus let $E$ be a given polynomial representation of $\Gamma_x$. By definition, there is some polynomial representation $V$ of $\GL$ and $\Gamma_x$-subrepresentations $E_2 \subset E_1 \subset V$ such that $E \cong E_1/E_2$. By Lemma~\ref{lem:genstab}, there exist $K$-submodules $W_2 \subset W_1 \subset K \otimes V$ such that $E_i=\Psi_x(W_i)$. Thus $E \cong \Psi_x(W_2/W_1)$, and so $\Psi_x$ is essentially surjective.

We now prove that $\Psi_x$ is full. Let $V$ and $W$ be $K$-modules and let $f \colon \Psi_x(V) \to \Psi_x(W)$ be a map of $\Gamma_x$-subrepresentations. Let $E \subset \Psi_x(V) \oplus \Psi_x(W)$ be the graph of $f$. By Lemma~\ref{lem:genstab}, we have $E=\Psi_x(U)$ for a unique $K$-submodule $U \subset V \oplus W$. The projection map $U \to V$ becomes an isomorphism after applying $\Psi_x$, and is therefore an isomorphism since $\Psi_x$ is exact and faithful. Thus $U$ is the graph of a morphism $g \colon V \to W$ of $K$-modules, and clearly $f=\Psi_x(g)$.

We have already seen that $\Psi_x$ is faithful, and so it is an equivalence.
\end{proof}

\end{document}